\RequirePackage{fix-cm}
\documentclass[smallextended]{svjour3}       
\smartqed  
\usepackage{amsfonts,amssymb,amsmath}
\usepackage{graphicx}
\usepackage{epstopdf}
\usepackage{tikz}
\usepackage{pgfplots}
\usepgfplotslibrary{fillbetween}
\pgfplotsset{compat=newest,
every axis/.append style={axis x line=bottom,
                          axis y line=left,
                          scale only axis,
                          y label style={at={(0.0,1.0)},anchor=south west,rotate=-90}
                          },
}
\usepackage[textsize=small]{todonotes}

\usepackage{algorithm,algpseudocode}
\ifpdf
  \DeclareGraphicsExtensions{.eps,.pdf,.png,.jpg}
\else
  \DeclareGraphicsExtensions{.eps}
\fi

\title{Tensor train solution to uncertain optimization problems with shared sparsity penalty\thanks{HA is partially supported by NSF grant DMS-2408877, the AirForce Office of Scientific Research under Award NO: FA9550-22-1-0248, and Office of Naval Research (ONR) under Award NO: N00014-24-1-2147.
SD is thankful for the support from Engineering and Physical Sciences Research
Council (EPSRC) New Investigator Award EP/T031255/1.}}

\author{Harbir Antil
\and
Sergey Dolgov
\and
Akwum Onwunta
}

\institute{H. Antil \at
              The Center for Mathematics and Artificial Intelligence
(CMAI) and Department of Mathematical Sciences, George Mason University,
Fairfax, VA 22030, USA. \\
              \email{hantil@gmu.edu}           
           \and
           S. Dolgov \at
              Department of Mathematical Sciences, University of Bath, Bath, BA2 7AY, UK. \\
              \email{s.dolgov@bath.ac.uk}
           \and
           A. Onwunta \at
           Department of Industrial and Systems Engineering, Lehigh University, Bethlehem, PA 18015, USA. \\
           \email{ako221@lehigh.edu}
}

\begin{document}
\maketitle

\begin{abstract}
We develop both first and second order numerical optimization methods to solve non-smooth optimization problems featuring a shared sparsity penalty, constrained by differential equations with uncertainty.
To alleviate the curse of dimensionality we use tensor product approximations.
To handle the non-smoothness of the objective function we employ a smoothed version of the shared sparsity objective.
We consider both a benchmark elliptic PDE constraint, and a more realistic topology optimization problem in engineering.
We demonstrate that the error converges linearly in iterations and the smoothing parameter, and faster than algebraically in the number of degrees of freedom, consisting of the number of quadrature points in one variable and tensor ranks.
Moreover, in the topology optimization problem, the smoothed shared sparsity penalty actually reduces the tensor ranks compared to the unpenalised solution. This enables us to find a sparse high-resolution design under a high-dimensional uncertainty.
\end{abstract}


\section{Introduction}

Constrained optimization is a standard approach in optimal design or control
where one needs to achieve a certain state of a mathematical model of a physical or technological process, often written in the form of partial differential equations (PDEs) nowadays, by minimizing a certain cost function, either fitting data or measuring the desired state (e.g. rigidity) in some other way. See e.g. the monographs \cite{Lions71,book::hpuu09,HAntil_DPKouri_MDLacasse_DRidzal_2018a} for further generic reading on this topic.
However, it becomes increasingly important to quantify the uncertainty emerging from the model and/or data, to both obtain and certify a control that is resilient to uncertainty, eventually paving the way towards `digital twins' \cite{HAntil_2024a,FAiraudo_HAntil_RLoehner_URakhimov_2024a}.
Such optimization problems are starting to receive a tremendous amount of attention.
There have been developments in all directions: modeling \cite{Chen_Haberman_Ghattas_2021}, analysis \cite{Milz_2023,Martin_Nobile_2021}, numerical
analysis \cite{Ali_Ullmann_Hinze_2017,Chen_Quarteroni_2014,Garreis_Ulbrich_2017,DPK2018}, numerical linear algebra \cite{Benner_Onwunta_Stoll_2016}, and algorithms \cite{Garreis_Ulbrich_2017}.

We consider an abstract PDE
\begin{equation}\label{eq:PDE_intro}
 c(y,u,\xi(\omega))=0,
\end{equation}
where $y \in \mathcal{Y}$ is the \emph{state} (function) belonging to a Hilbert state space $\mathcal{Y}$,
and $u\in\mathcal{U}_{ad}$ is the \emph{control} (function) belonging to the admissible control subset $\mathcal{U}_{ad} \subset \mathcal{U}$ of a Hilbert space $\mathcal{U}$.
The left hand side $c(y,u,\cdot) \in \mathcal{Z}$ where $\mathcal{Z}$ is a residual Hilbert space.
The state $y$ and control $u$ are maps from $D \times \Xi$, where
$D\subset \mathbb{R}^{\hat d}$ is a physical domain with Lipschitz boundary $\partial D$ (such as a volume in space), and $\Xi\subset \mathbb{R}^d$ is the image of the random vector $\xi(\omega) : \Omega \mapsto \Xi$.
The physical variable is denoted $x\in D$.
The admissible control set $\mathcal{U}_{ad} $ is assumed
to be closed, convex, and nonempty.
The random vector $\xi(\omega)$ models the uncertainty in the PDE model~\eqref{eq:PDE_intro}, usually its coefficients (e.g. viscosity), and is defined through a complete
probability space $(\Omega,\mathfrak{F},\mathbb{P})$, where $\Omega$ is a space of all possible outcomes,
$\mathfrak{F}\subset 2^{\Omega}$ is a $\sigma$-algebra on $\Omega$ and
$\mathbb{P}:\mathfrak{F}\rightarrow [0,1]$ is an appropriate probability measure.
Particular examples of such PDE models with uncertainty can be found in numerical experiments, Sections~\ref{sec:elliptic} and \ref{sec:topology}.
For simplicity of exposition, we make a finite-dimension noise assumption
encapsulated by the map $\xi$ from $\Omega$ to a finite dimensional domain $\Xi
= \xi(\Omega) \subset \mathbb{R}^d$ and a continuous probability density function (PDF)
$\rho(\xi): \Xi \rightarrow \mathbb{R}_+$.
The assumption of existence of the PDF allows us to express expectations of random vectors (and functions thereof) as integrals
$$
\mathbb{E}[g(\xi(\omega))] := \int_{\Xi} g(\xi) \rho(\xi) d\xi, \qquad g: \Xi \rightarrow \mathbb{R}.
$$
The decision making becomes a problem of optimizing a cost function,
\begin{equation}\label{eq:obj_intro}
	\min_{y\in\mathcal{Y},u\in \mathcal{U}_{ad}} \mathcal{J}(y,u), \quad
		 \mathcal{J}(y,u) := \mathbb{E}[J(y,\xi)]
		+ \mathbb{E}[S(u)]
		+ \mathcal{R}(u),
\end{equation}
constrained by \eqref{eq:PDE_intro}. Here, $J$ is the random-variable objective, $S$ and $\mathcal{R}$ respectively denote smooth and nonsmooth control regularizations.
For example, $S$ can be the standard squared $L^2$-norm of $u$ as used in the Tikhonov regularization, or other smooth (at least, continuously differentiable) function, whereas $\mathcal{R}$ is a nonsmooth regularization term, such as the $L^1$-norm of $u$, or the shared sparsity penalty we consider in this paper.

There exist two versions of the uncertain PDE-constrained optimal control problem \eqref{eq:PDE_intro}--\eqref{eq:obj_intro}.
Firstly, we can introduce random variables in the PDE, but seek a deterministic control $u(x)$ \cite{KHRW13,DPK2018,Garreis_Ulbrich_2017,KS16}. 
A straightforward objective function of the state thus depends on the random variables too.
Turning it into a deterministic total cost needs an expectation of either the objective itself (risk-neutral optimal control problem) or another function of the objective (risk-averse problem).
Such controls are easy to realize, but may be restrictive or lack robustness.
Secondly, we can introduce random variables in both state $y(x,\xi)$ and control $u(x,\xi)$ \cite{Chen_Quarteroni_2014,Benner_Onwunta_Stoll_2016,Tiesler_et_al_2012,Negri_et_al_2013}.
This provides uncertainty quantification for control,
but it is not immediately clear how to realize a random field control in practice.

An interesting compromise was introduced in \cite{Li_Stadler_2019}, which develops the notion of
`shared-sparsity'.
The idea is to consider an $L^1$-norm in the spatial (physical)
domain and an $L^2$-norm in the random variable, i.e.,
\begin{equation}\label{eq:R0}
	\mathcal{R}(u) = \beta \int_D \left(\int_{\Xi} |u(x,\xi)|^2 \rho(\xi) d\xi \right)^{\frac12} dx \, ,
\end{equation}
where $\beta \ge 0$ is the regularization parameter.
The optimization with a 1-norm penalty to promote sparsity (that is, to drive near-zero coefficients or gradients to be exactly zero) is a well-known technique in inverse problems \cite{Vogel-inv,Milz_2023} and optimal control problems \cite{GStadler_2009a,MR2914236}.
Here, we make the sparsity of the control (almost) deterministic, shared over all the realizations of the random variable.
Indeed, if $u(x,\xi)$ is nonzero on any set of a positive measure, the entire integral over $\xi$ in \eqref{eq:R0} is positive, and hence contributes a penalty to $\mathcal{R}(u).$
In turn, the optimal value of the latter is either sufficiently large or zero due to the $L^1$ norm.

On the other hand, an approximation of the entire random field $u(x,\xi)$ can be precomputed without knowing any actual state of the system. For this reason, these computations are called the \emph{offline stage}.
In contrast, the actual control loop where the precomputed $u(x,\xi)$ is evaluated and applied to the system is called the \emph{online stage}.
Precomputation of an approximate $u(x,\xi)$ in a compressed format significantly reduces the computational burden of evaluating the actual control signal, which is important for real-time control.
How exactly this signal is produced may vary.
One can apply the mean or quantile of $u(x,\xi)$, or estimate a particular $\xi$ from state measurements using e.g. Bayesian filtering \cite{jazwinski-stochastic-2007,arulampalam-tutorial-2002} and apply the control evaluated on this $\xi$.
Crucially, since the positions of the zeros in the control are (almost) deterministic, one can fix the actuators to the rest of the spatial domain independently of the system state.

We consider the optimization of topology (the layout of a material within a given design space ensuring certain properties such as rigidity) under uncertain Young's modulus of the material as a particularly interesting application of the `shared-sparsity' concept.
The goal is to find an optimal distribution of the material
of a total relative volume $\bar{V}$ over a given domain $D$ by minimizing the compliance
of a structure subject to elasticity equations as constraints. In this context, finding sparsity
corresponds to identifying locations where one must add material. Then in the online
stage, the structure can be designed using different realizations of the random variables,
all producing a robust design. Topology optimization under uncertainty have
received some attention recently \cite{Keshavarzzadeh-top-2021,Torres-topology-2021,duswald2024finite}, but these references do not focus on sparsity.

The main difficulty of uncertainty quantification is the curse of dimensionality: the number of random variables $d$ needed to parametrize the sought random field can be tens to thousands.
Direct discretization of each variable independently produces the number of unknowns growing exponentially in $d$, quickly becoming unfeasible.
Except for particularly low $d$ or smooth functions,
where one can use sparse grids for instance \cite{Tiesler_et_al_2012,Negri_et_al_2013}, almost all
the existing algorithms in the nonsmooth setting use Monte-Carlo based sampling, which may converge rather slowly and lead to a sheer computational complexity due to the need to solve a lot of PDEs during the optimization.

Another approach that can approximate high-dimensional functions (such as discretized random fields) is hierarchical tensor decompositions~\cite{hackbusch-2012}.
The idea is that we discretize a function using a Cartesian product of univariate basis functions, but instead of storing the tensor of expansion coefficients explicitly, we approximate it by a sum of products of low-dimensional factors.
Smooth functions~\cite{khor-low-rank-kron-P1-2006,uschmajew-approx-rate-2013}, or probability density functions of weakly correlated random variables~\cite{rdgs-tt-gauss-2022}, were proven to admit low tensor ranks of such approximations, e.g. depending poly-logarithmically on the approximation error.
Moreover, hierarchical tensor decompositions (specifically, the Tensor-Train (TT) decomposition~\cite{osel-tt-2011} we use in this paper) are equipped with fast adaptive cross approximation algorithms~\cite{ot-ttcross-2010,ds-parcross-2020}.

However, tensor approximations struggle with non-smooth functions.
To apply them to non-smooth optimization, 
\cite{Antil_Dolgov_Onwunta_CVaR2023,Antil_Dolgov_Onwunta_Statecon2023}
introduced smoothed objective functions such that the solution of the smoothed problem converges to the solution of the original problem as the smoothing parameter goes to zero,
whereas for any positive smoothing parameter the solution exhibits a rapidly converging TT approximation.

Similarly to \cite{Li_Stadler_2019}, we introduce
an $\varepsilon$-smoothing of the shared sparsity objective~\eqref{eq:R0}. Specifically, for any $\varepsilon>0$ we replace~\eqref{eq:R0} by
\begin{equation}\label{eq:Reps}
	\mathcal{R}_\varepsilon(u) = \beta \int_D \left(\int_{\Xi} u(x,\xi)^2 \rho(\xi) d\xi + \varepsilon^2 \right)^{\frac12} dx \,
\end{equation}
in the objective function~\eqref{eq:obj_intro}.

\subsection{Main contributions}
\begin{itemize}
 \item We develop an approximate Hessian of \eqref{eq:Reps} that can be approximated efficiently by using the TT Cross algorithm. This enables a fast Newton-type optimization method for objective functions with the shared-sparsity penalty.
 \item We benchmark both computational and modelling effects of the shared-sparsity penalty using the TT approximations on both an academic example of an elliptic PDE to study different aspects of the methodology, and a more realistic and large-scale application of the topology optimization under uncertainty.
\end{itemize}

The remainder of the article is organized as follows.
Section~\ref{sec:opt} details the smoothed shared sparsity objective, its derivatives, optimality 
conditions in both continuous and discrete settings, 
and the fully discrete approximate Newton's method with the fast application of the Hessian.
This formalism is independent of the particular approximation of functions and expectations over the random parameters.
However, for a practical method free from the curse of dimensionality,
we overview the Tensor-Train decomposition in Section~\ref{sec:tt}.
In Section~\ref{sec:num} we test our methodology using a couple of numerical examples: an elliptic PDE in Section~\ref{sec:elliptic} for faster computations and more illustrative behaviour of the method, and a robust topology optimization problem~\cite{Keshavarzzadeh-top-2021,Torres-topology-2021,Audouze-top-2023} in Section~\ref{sec:topology}.

\section{Discretization and optimality conditions}\label{sec:opt}
We assume that both $y(x,\xi)$ and $u(x,\xi)$ are approximated by $y_h(x,\xi)$ and $u_h(x,\xi)$ which are expanded by $\hat N$ basis functions $\psi_1(x),\ldots,\psi_{\hat N}(x)$ in the physical domain, discretised with a grid of the maximal element diameter $h>0$.
Without loss of generality, we can assume that the same basis functions are used for $y$ and $u$ (we can always concatenate different bases and expand both functions in the joint basis).
For the TT decomposition method we further assume that $y_h(x,\xi)$ and $u_h(x,\xi)$ are expanded by $N$ basis functions $\phi_1(\xi), \ldots, \phi_N(\xi)$ on the random variable.
Ultimately,
$$
y_h(x,\xi) = \sum_{i=1}^{\hat N} \sum_{j=1}^N y_{i,j} \psi_i(x) \phi_j(\xi), \qquad u_h(x,\xi) = \sum_{i=1}^{\hat N} \sum_{j=1}^N u_{i,j} \psi_i(x) \phi_j(\xi),
$$
and we can collect coefficients into vectors
\begin{equation}\label{eq:discr-vec}
\mathbf{y} = \left[y_{1,1}, \ldots, y_{\hat N, N}\right], \quad \mathbf{u} = \left[u_{1,1}, \ldots, u_{\hat N, N}\right].
\end{equation}
For Monte Carlo methods, we assume that $\xi$ is sampled in $N$ independent identically distributed points $\xi_1,\ldots,\xi_N$, so we can
expand
$$
y_h(x,\xi_j) = \sum_{i=1}^{\hat N} y_{i,j} \psi_i(x), \qquad u_h(x,\xi_j) = \sum_{i=1}^{\hat N} u_{i,j} \psi_i(x), \qquad j=1,\ldots,N,
$$
and collect the coefficients into vectors similarly to \eqref{eq:discr-vec}.

\subsection{Full space formulation}
First, we consider the Lagrangian formulation.
Let $\lambda(x,\xi) \in \mathcal{Z}^*$ be the Lagrange multiplier belonging to the dual of the residual Hilbert space $\mathcal{Z} \ni c(y,u,\cdot)$, so we can identify $\mathcal{Z}^*$ with $\mathcal{Z}$.
Then we need to find the KKT conditions of the Lagrangian
\begin{equation}\label{eq:lagr-cont}
\mathcal{L}(y,u,\lambda) := \mathbb{E}[J(y,\xi)]
		+ \mathbb{E}[S(u)]
		+ \mathcal{R}_{\varepsilon}(u) - \int_D \int_{\Xi} \lambda(x,\xi) c(y,u,\xi)\rho(\xi) d\xi dx.
\end{equation}
We assume that the discretized Lagrange multiplier $\lambda_h(x,\xi)$ is expanded in the same basis as $y_h$ and $u_h$.
Plugging in the discrete solutions into \eqref{eq:lagr-cont}, we obtain $\mathcal{L}_h(\mathbf{y},\mathbf{u},\boldsymbol\lambda) := \mathcal{L}(y_h,u_h,\lambda_h)$. Differentiating the latter with respect to the coefficient vectors $\mathbf{y},\mathbf{u},\boldsymbol\lambda \in \mathbb{R}^{\hat N N}$,
we get the state and adjoint equations, and an expression of the gradient
\begin{align}\label{eq:grad_y_lagr_discr}
(\nabla_{\mathbf{y}} \mathcal{L}_h)_{i,j} & = \int_D \int_{\Xi} \Big(\nabla_{y} J(y_h,\xi)  -   \lambda_h \nabla_y c(y_h,u_h,\xi)\Big) \psi_i \phi_j \rho d\xi dx = 0, \\
(\nabla_{\boldsymbol\lambda} \mathcal{L}_h)_{i,j} & = - \int_D \int_{\Xi} c(y_h,u_h,\xi) \psi_i \phi_j \rho d\xi dx = 0, \label{eq:grad_p_lagr_discr} \\
(\nabla_{\mathbf{u}} \mathcal{L}_h)_{i,j} & = \int_D \int_{\Xi} \Big(\nabla_{u} S(u_h) - \lambda_h \nabla_u c(y_h,u_h,\xi)\Big) \psi_i \phi_j \rho d\xi dx + (\nabla_{\mathbf{u}} \mathcal{R}_{\varepsilon})_{i,j}, \label{eq:grad_u_lagr_discr}
\end{align}
where $\nabla_y$ and $\nabla_u$ in the right-hand-side are Fr\'echet derivatives, in contrast to the gradient with respect to discrete coefficients in the left-hand-side.
The latter finite-dimensional gradients will be used for the TT method.
We emphasize that $(\nabla_{\mathbf{u}} \mathcal{L}_h)_{i,j}$ may not be equal to zero in general due to the constraints on the control. Instead, $(\nabla_{\mathbf{u}} \mathcal{L}_h)_{i,j}$ fulfills a standard variational inequality, see for instance \cite[Ch.~1, Theorem~2]{HAntil_DPKouri_MDLacasse_DRidzal_2018a}.
For the Monte Carlo method, we replace the integrals over $\xi$ in the Lagrangian by the averages over samples at $\xi_1,\ldots,\xi_N$:
\begin{align*}
& \mathcal{L}_N(\mathbf{y},\mathbf{u},\boldsymbol\lambda) := \\
& \frac{1}{N} \sum_{j=1}^N \left[ J(y_h,\xi_j)
        + S(u_h(x,\xi_j))
        - \int_D \lambda_h(x,\xi_j) c(y_h,u_h,\xi_j) dx \right]
        + \mathcal{R}_{\varepsilon,N}(\mathbf{u}),
\end{align*}
where
\begin{equation}\label{eq:RepsN}
\mathcal{R}_{\varepsilon,N}(\mathbf{u}) = \beta \int_D \left(\frac{1}{N}\sum_{j=1}^N u_h(x,\xi_j)^2 + \varepsilon^2 \right)^{\frac12} dx.
\end{equation}
Since $\xi_j$ for different $j$ are independent, and so are $u_{i,j}$ and $y_{i,j}$,
the sum over $j$ collapses in the following gradients:
\begin{align*}
(\nabla_{\mathbf{y}} \mathcal{L}_N)_{i,j} & = \frac{1}{N}\left[(\nabla_{\mathbf{y}} J)_{i,j} - \int_D \lambda_h(x,\xi_j) (\nabla_y c(y_h,u_h,\xi_j)) \psi_i dx\right], \\
(\nabla_{\boldsymbol\lambda} \mathcal{L}_N)_{i,j} & = -\frac{1}{N} \int_D c(y_h,u_h,\xi_j) \psi_i dx, \\
(\nabla_{\mathbf{u}} \mathcal{L}_N)_{i,j} & = \frac{1}{N}\left[(\nabla_{\mathbf{u}} S)_{i,j}  - \int_D \lambda_h(x,\xi_j) (\nabla_u c(y_h,u_h,\xi_j)) \psi_i dx \right] + (\nabla_{\mathbf{u}} \mathcal{R}_{\varepsilon,N})_{i,j}.
\end{align*}

\subsection{Derivatives of the shared sparsity penalty}
Both first and second derivatives of $J,S$ and $c$ will be instantiated in concrete examples to simplify their derivation.
Here, we are concerned with the derivatives of \eqref{eq:Reps} and \eqref{eq:RepsN}.

For the TT method, plugging in the discretized solution into \eqref{eq:Reps}, we obtain
\begin{align}
\mathcal{R}_{\varepsilon}(u_h) & = \beta \int_D \left(\int_{\Xi} \left(\sum_i \sum_j u_{i,j} \psi_i(x) \phi_j(\xi)\right)^2 \rho(\xi) d\xi + \varepsilon^2 \right)^{\frac12} dx, \\
(\nabla_{\mathbf{u}}\mathcal{R}_{\varepsilon})_{i,j} & = \beta \int_D \frac{\int_{\Xi} \left(\sum_{i'} \sum_{j'} u_{i',j'} \psi_{i'}(x) \phi_{j'}(\xi)\right) \psi_i(x) \phi_j(\xi) \rho(\xi) d\xi}{\sqrt{\int_{\Xi} \left(\sum_{i'} \sum_{j'} u_{i',j'} \psi_{i'}(x) \phi_{j'}(\xi)\right)^2 \rho(\xi) d\xi + \varepsilon^2}} dx.
\end{align}
In practice, we can always use the Lagrange interpolation basis functions $\phi_1(\xi),\ldots,\phi_N(\xi)$, centered at nodes $\xi_1,\ldots,\xi_N$ of a multivariate Gaussian grid, such that $\phi_j(\xi_{j'}) = \delta_{j,j'},$
where $\delta_{j,j'}$ is the Kronecker symbol equal to $1$ when $j=j'$ and $0$ otherwise.
The integral over $\xi$ can then be approximated by a quadrature,
$$
\int_{\Xi} f(\xi) \rho(\xi) d\xi \approx \sum_{j=1}^{N} w_j f(\xi_j),
$$
where $w_1,\ldots,w_N$ are quadrature weights.
Moreover, the Gaussian quadrature is exact for polynomials of degree up to $2N-1$.
This gives
$$
\int_{\Xi} \left(\sum_i \sum_j u_{i,j} \psi_i(x) \phi_j(\xi)\right)^2 \rho(\xi) d\xi = \sum_{j=1}^{N} w_j \left[\sum_{i,i'=1}^{\hat N} \left[u_{i,j} u_{i',j} \psi_i(x) \psi_{i'}(x)\right]\right].
$$

The integral over $x$ is taken of a non-polynomial function, and needs a quadrature too.
The usual finite element practice is to take $\psi_1(x),\ldots,\psi_{\hat N}(x)$ to be \emph{nodal} on a grid $x_1,\ldots,x_N$, such that $\psi_i(x_{i'}) = \delta_{i,i'},$
and $\psi_i(x)$ is zero outside of some neighbourhood of $x_i$.
We can introduce a quadrature on the same nodes,
$$
\int_{D} f(x) dx \approx \sum_{i=1}^{\hat N} \hat w_i f(x_i).
$$
For example, the Trapezoidal rule with piecewise linear basis functions satisfies this assumption and gives the second order of approximation.
This gives
\begin{align}
\mathcal{R}_{\varepsilon,h}(\mathbf{u}) & = \beta \sum_{i=1}^{\hat N} \hat w_i \sqrt{ \sum_{j=1}^{N} \left[w_j u_{i,j}^2\right] + \varepsilon^2}, \\ \label{eq:RepsGrad}
(\nabla_{\mathbf{u}}\mathcal{R}_{\varepsilon,h})_{i,j} & = \beta \hat w_i \frac{u_{i,j} w_j}{\sqrt{ \sum_{j'=1}^{N} \left[w_{j'} u_{i,j'}^2\right] + \varepsilon^2}}, \\
(\nabla^2_{\mathbf{u}\mathbf{u}}\mathcal{R}_{\varepsilon,h})_{i,j,i',j'} & = \frac{\beta \hat w_i w_j \delta_{j,j'} \delta_{i,i'}}{\sqrt{ \sum_{j''=1}^{N} \left[w_{j''} u_{i,j''}^2\right] + \varepsilon^2}} - \frac{\beta \hat w_i w_j  w_{j'} u_{i,j} u_{i,j'} \delta_{i,i'}}{\left(\sum_{j''=1}^{N} \left[w_{j''} u_{i,j''}^2\right] + \varepsilon^2\right)^{\frac32}}. \label{eq:RepsHess}
\end{align}
For the Monte Carlo method, $\nabla_{\mathbf{u}}\mathcal{R}_{\varepsilon,N}$ and $\nabla^2_{\mathbf{u}\mathbf{u}}\mathcal{R}_{\varepsilon,N}$ have the same form as \eqref{eq:RepsGrad} and \eqref{eq:RepsHess}, respectively, with $w_j = 1/N$.

\subsection{Linear-Quadratic problem}
Practical expressions are more convenient to write assuming linear constraints and quadratic costs.
Let us assume that the spaces $\mathcal{Y},\mathcal{Z}$ and $\mathcal{U}$ are Bochner spaces
$\mathcal{Y} = L^2_{\rho}(\Xi; Y),$
$\mathcal{Z} = L^2_{\rho}(\Xi; Z),$
and
$\mathcal{U} = L^2_{\rho}(\Xi; U)$,
where $Y,Z$ and $U$ are Hilbert spaces of functions of $x$ only (for example, $L^2(D)$),
and
$L^2_{\rho}$ is a space of functions of $\xi$ with the $L^2$ inner product weighted with the probability density function $\rho(\xi)$.
Now we can assume that
\begin{align*}
J(y,\xi) & = \frac{1}{2}\|y - y_{d}(x,\xi)\|^2_{M_y}, &
S(u) & = \frac{\alpha}{2}\|u\|^2_{M_u},  \\
c(y,u,\xi) & = A(\xi) y - B(\xi) u - f(\xi),
\end{align*}
where $M_y: Y \rightarrow Y$ is a self-adjoint positive semi-definite operator,
$y_d(x,\xi) \in \mathcal{Y}$ is a desired state,
$M_u: U \rightarrow U$ is a self-adjoint positive definite operator,
$\|f\|_M := \sqrt{\langle f, M f \rangle_{L^2(D)}}$ is the induced $M$-norm,
$\alpha\ge 0$ is a regularization parameter,
$A(\xi): Y \rightarrow Z$ for any $\xi \in \Xi$ is a linear PDE operator on the state,
$B(\xi): U \rightarrow Z$ is a linear actuator operator of the control,
and $f(\xi) \in Z$ is an uncontrollable right hand side of the PDE.
In this case,
\begin{align*}
\nabla_{y} J(y_h,\xi) & = M_y \left(y_h(x,\xi) - y_d(x,\xi)\right), \\
\nabla_{u} S(u_h) & =  \alpha M_u u_h(x,\xi), \\
\nabla_y c & = A(\xi), \\
\nabla_u c & = -B(\xi).
\end{align*}
To assemble \eqref{eq:grad_y_lagr_discr}--\eqref{eq:grad_u_lagr_discr},
let $\langle \cdot, \cdot \rangle_{V,W}$ be the standard duality pairing between Banach spaces $V$ and $W$, and
introduce the mass matrices with elements
\begin{align*}
(\mathbf{M}_y)_{i,i'} & = \langle \psi_i,  M_y \psi_{i'} \rangle_{Y,Y}, & (\mathbf{M}_u)_{i,i'}
& = \langle \psi_i,  M_u \psi_{i'} \rangle_{U,U},
\end{align*}
the desired state vector $(\mathbf{y}_d(\xi))_i = \langle \psi_i, y_d(\cdot,\xi) \rangle_{Y,Y}$,
the stiffness matrices with elements
\begin{align*}
(\mathbf{A}(\xi))_{i,i'} & = \langle \psi_i,  A(\xi) \psi_{i'}\rangle_{Y,Z}, & (\mathbf{B}(\xi))_{i,i'} & = \langle \psi_i,  B(\xi) \psi_{i'}\rangle_{Y,Z},
\end{align*}
and the right hand side $(\mathbf{f}(\xi))_i = \langle \psi_i, f(\cdot,\xi)\rangle_{Y,Z}$.
Notice that for simplicity of presentation, we are using the same basis functions $\{\psi_i\}$ to
discretize $Y$, $U$, and $Z$ spaces.
Moreover, an appropriate weak formulation can be introduced for the stiffness matrices.
Replacing the integral over $\xi$ by the quadrature,
we can write the fully discrete Lagrangian gradient
\begin{align}
(\nabla_{\mathbf{y}} \mathcal{L}_h)_{i,j} & = w_j \sum_{i'=1}^{\hat N} \left[(\mathbf{M}_y)_{i,i'} y_{i',j}  - \lambda_{i',j} (\mathbf{A}(\xi_j))_{i',i}\right] - w_j(\mathbf{y}_d(\xi_j))_{i}, \\
(\nabla_{\mathbf{u}} \mathcal{L}_h)_{i,j} & = w_j \sum_{i'=1}^{\hat N}\left[\alpha(\mathbf{M}_u)_{i,i'} u_{i',j} + \lambda_{i',j} (\mathbf{B}(\xi_j))_{i',i}\right] + (\nabla_{\mathbf{u}}\mathcal{R}_{\varepsilon,h})_{i,j}, \\
(\nabla_{\boldsymbol\lambda} \mathcal{L}_h)_{i,j} & = w_j \sum_{i'=1}^{\hat N}\left[-(\mathbf{A}(\xi_j))_{i,i'} y_{i',j} + (\mathbf{B}(\xi_j))_{i,i'} u_{i',j} + (\mathbf{f}(\xi_j))_i \right]
\end{align}
to be used for the TT method.
The expressions for the Monte Carlo method are the same replacing $w_j$ by $1/N$ and $\mathcal{R}_{\varepsilon,h}$ by $\mathcal{R}_{\varepsilon,N}$.
To write the Hessian in a more convenient matrix form,
let
\begin{align*}
\boldsymbol A &= \begin{bmatrix}\mathbf{A}(\xi_1) \\ & \ddots \\ & & \mathbf{A}(\xi_N)\end{bmatrix} \in\mathbb{R}^{N\hat N \times N \hat N},
&
\boldsymbol B &= \begin{bmatrix}\mathbf{B}(\xi_1) \\ & \ddots \\ & & \mathbf{B}(\xi_N)\end{bmatrix} \in\mathbb{R}^{N\hat N \times N \hat N}
\end{align*}
be block-diagonal matrices,
\begin{align*}
\boldsymbol W & = \begin{bmatrix}w_1 \\ & \ddots \\ & & w_N\end{bmatrix} \otimes I_{\hat N}, &
\boldsymbol f & = \begin{bmatrix}\mathbf{f}(\xi_1) \\ \vdots \\ \mathbf{f}(\xi_N)
\end{bmatrix},
\quad
\boldsymbol y_d = \begin{bmatrix}\mathbf{y}_d(\xi_1) \\ \vdots \\ \mathbf{y}_d(\xi_N)
\end{bmatrix},
\end{align*}
where $I_m \in \mathbb{R}^{m \times m}$ is the identity matrix, and $\otimes$ is the Kronecker product, then
$$
\nabla^2 \mathcal{L}_h = \begin{bmatrix}
\boldsymbol W (I_N \otimes \mathbf{M}_y) & 0 & -\boldsymbol W \boldsymbol A^\top \\
0 & \boldsymbol W(\alpha I_N \otimes \mathbf{M}_u) + \nabla^2_{\mathbf{u}\mathbf{u}}\mathcal{R}_{\varepsilon,h} & \boldsymbol W\boldsymbol B^\top \\
-\boldsymbol W\boldsymbol A & \boldsymbol W\boldsymbol B & 0
\end{bmatrix},
$$
with a similar replacement of $w_j=1/N$ for the Monte Carlo Hessian $\nabla^2 \mathcal{L}_N.$

\subsection{Approximate block-diagonal Hessian and Newton's method}
All terms in the Hessian above are block-diagonal with respect to $j$ except $\nabla^2_{\mathbf{u}\mathbf{u}}\mathcal{R}_{\varepsilon,h}$.
In turn, only the second term in \eqref{eq:RepsHess} breaks this block-diagonality.
Moreover, while both terms of \eqref{eq:RepsHess} have the same asymptotic scaling in $u$ ($\mathcal{O}(1/|u|)$ for $|u_{i,j}| \ge |u| \gg \varepsilon$ and $\mathcal{O}(1/\varepsilon)$ for $|u_{i,j}| \ll \varepsilon$),
the second term has an additional factor $w_{j'}<1$, since $w_{j}>0$ and $\sum_j w_j=1$ for the probability normalization.
In fact, $w_{j'} = 1/N$ for the Monte Carlo method, and $w_{j'} \le C/N \ll 1$ for the TT method, since typically $N \gg 1$ for both TT and Monte Carlo.
Thus, we omit the second term of \eqref{eq:RepsHess}, and approximate $\nabla^2_{\mathbf{u}\mathbf{u}}\mathcal{R}_{\varepsilon,h}$ by $\tilde \nabla^2_{\mathbf{u}\mathbf{u}}\mathcal{R}_{\varepsilon,h}$ whose components are given by
\begin{align}
(\tilde \nabla^2_{\mathbf{u}\mathbf{u}}\mathcal{R}_{\varepsilon,h})_{i,j,i',j'} & = \frac{\beta \hat w_i \delta_{j,j'} \delta_{i,i'}}{\sqrt{ \sum_{j''=1}^{N} \left[w_{j''} u_{i,j''}^2\right] + \varepsilon^2}}, \\
\tilde \nabla^2 \mathcal{L}_h & = (I_3 \otimes \boldsymbol W)\begin{bmatrix}
I_N \otimes \mathbf{M}_y & 0 & - \boldsymbol A^\top \\
0 & \alpha I_N \otimes \mathbf{M}_u + \tilde \nabla^2_{\mathbf{u}\mathbf{u}}\mathcal{R}_{\varepsilon,h} & \boldsymbol B^\top \\
-\boldsymbol A & \boldsymbol B & 0
\end{bmatrix}.\label{eq:HessApprox}
\end{align}
Note that now all terms contain $w_j$ in front of them,
so we can take $\boldsymbol W$ out of the Hessian,
which will cancel with $\boldsymbol W$ in front of the gradient
\begin{align}\label{eq:GradYVec}
\nabla_{\mathbf{y}} \mathcal{L}_h & = \boldsymbol W \left[(I_N \otimes \mathbf{M}_y)  \mathbf{y}  - \boldsymbol A^\top \boldsymbol\lambda - \boldsymbol y_d\right], \\
\nabla_{\mathbf{u}} \mathcal{L}_h & = \boldsymbol W \left[(\alpha I_N \otimes \mathbf{M}_u) \mathbf{u} + \boldsymbol B^\top \boldsymbol\lambda + \tilde \nabla_{\mathbf{u}}\mathcal{R}_{\varepsilon,h} \right], \\
\nabla_{\boldsymbol\lambda} \mathcal{L}_h & =  \boldsymbol W \left[-\boldsymbol A \mathbf{y} + \boldsymbol B \mathbf{u} + \boldsymbol f \right],  \label{eq:GradLVec}
\end{align}
where a gradient without $w_j$ reads
$$
(\tilde \nabla_{\mathbf{u}}\mathcal{R}_{\varepsilon,h})_{i,j}  = \beta \hat w_i \frac{u_{i,j}}{\sqrt{ \sum_{j'=1}^{N} \left[w_{j'} u_{i,j'}^2\right] + \varepsilon^2}}.
$$
This allows us to write Newton's method as summarized in Algorithm~\ref{alg:newton}.
Note that we rewrite the update as a linear equation on the next iterate directly.
This will be beneficial for the TT computations in Section~\ref{sec:cross}.
This also uses the step size of $1$.
For globalization, one can rewrite Steps~\ref{alg:newton:resid}--\ref{alg:newton:solve} in the usual increment form, followed by a line search, although the TT ranks of the increment $\mathbf{s}^{[k+1]}-\mathbf{s}^{[k]}$ can be larger than the TT ranks of $\mathbf{s}^{[k]}$ or $\mathbf{s}^{[k+1]}$ alone.
The linear solution (Step~\ref{alg:newton:solve})
can be implemented in the TT format in a problem-dependent way.
One example is shown in the next section.

\begin{algorithm}[htb]
\caption{Approximate Newton's method for the shared-sparsity optimization}
\label{alg:newton}
\begin{algorithmic}[1]
\Require Initial guess $\mathbf{y}^{[0]}, \mathbf{u}^{[0]}, \boldsymbol\lambda^{[0]}$, stopping tolerance $\mathrm{tol}>0$, maximum number of iterations $K$.
\Ensure Updated iterations $\mathbf{y}^{[k]}, \mathbf{u}^{[k]}, \boldsymbol\lambda^{[k]}$.
\For{$k=0,1,\ldots,K-1$}
  \State Let $\mathbf{s}^{[k]} = \begin{bmatrix}(\mathbf{y}^{[k]})^\top, (\mathbf{u}^{[k]})^\top, (\boldsymbol\lambda^{[k]})^\top\end{bmatrix}^\top$ be the previous solution vector.
  \State Assemble $\tilde \nabla^2 \mathcal{L}_h$ and $\mathbf{b} := \tilde \nabla^2 \mathcal{L}_h \mathbf{s}^{[k]} - \nabla \mathcal{L}_h$ with $\mathbf{y}=\mathbf{y}^{[k]}$, $\mathbf{u}=\mathbf{u}^{[k]}$ and $\boldsymbol\lambda=\boldsymbol\lambda^{[k]}$ as shown in \eqref{eq:HessApprox} and \eqref{eq:GradYVec}--\eqref{eq:GradLVec}. \label{alg:newton:resid}
  \State Solve $\tilde \nabla^2 \mathcal{L}_h \mathbf{s}^{[k+1]}  = \mathbf{b}$.\label{alg:newton:solve}
  \State Split the components $\begin{bmatrix}(\mathbf{y}^{[k+1]})^\top, (\mathbf{u}^{[k+1]})^\top, (\boldsymbol\lambda^{[k+1]})^\top\end{bmatrix}^\top = \mathbf{s}^{[k+1]}.$
  \If{$\|\mathbf{s}^{[k+1]} - \mathbf{s}^{[k]}\| \le \mathrm{tol} \cdot \|\mathbf{s}^{[k+1]}\|$}
    \State \textbf{break}.
  \EndIf
\EndFor
\end{algorithmic}
\end{algorithm}

\subsection{Reduced space formulation}
Assuming that for every $u \in \mathcal{U}_{ad}$ there is a unique
solution $y(u; x,\xi) \in \mathcal{Y}$ to \eqref{eq:PDE_intro}, we can introduce
a control-to-state map $\mathcal{U}_{ad} \ni u \mapsto
\Upsilon(u)(x,\xi) := y(u; x,\xi)$, and the reduced version of the cost
\eqref{eq:obj_intro}:
\begin{equation}\label{eq:RP}
	\min_{u\in \mathcal{U}_{ad}} j(u),
		\quad  j(u) := \mathbb{E}[J(\Upsilon(u),\xi)]
		+ \mathbb{E}[S(u)]
		+ \mathcal{R}_{\varepsilon}(u).
\end{equation}
The reduced cost and its derivatives can be discretized using the basis or Monte Carlo methods from above,
and the quadrature weights $w_j$ or $1/N$ can be cancelled in the Hessian and gradient.

Existence of the solution to~\eqref{eq:RP} and convergence of the smoothed approximation was established in~\cite{Li_Stadler_2019} for linear constraints.
\begin{proposition}[\cite{Li_Stadler_2019}, Theorem~4 and Corollary~5]
For any $\beta \ge 0,$ $\varepsilon \ge 0,$ and $\alpha>0$, if $A: \mathcal{Y} \subset H^1(D) \rightarrow H^{-1}(D)$ is linear and invertible, and $y_d \in \mathcal{Y}$,
there exists a unique solution to
\begin{equation}\label{eq:RP_lin}
\min_{u\in \mathcal{U}_{ad}} \int_D \left[\int_{\Xi} (A^{-1} (Bu+f) - y_d)^2 + \alpha u^2 \right]  dx \rho(\xi) d\xi + \mathcal{R}_{\varepsilon}(u).
\end{equation}
\end{proposition}
\begin{proposition}[\cite{Li_Stadler_2019}, Lemma~6]
Let $u_{\varepsilon}$ be the solution to \eqref{eq:RP_lin} with the smoothing parameter $\varepsilon \ge 0$,
and let $u_{0}$ be the solution to \eqref{eq:RP_lin} with $\varepsilon=0$.
If $D$ is bounded, then
$$
\|u_{\varepsilon} - u_0\|^2_{L^2_{\rho}(\Xi; L^2(D))} \le \varepsilon \beta \alpha^{-1} |D|.
$$
\end{proposition}
Taking the square root of both sides, we can notice that the upper bound of the error in the control decays proportionally to $\varepsilon^{1/2}.$

\section{Functional Tensor Train approximation}\label{sec:tt}
In the course of optimization iterations we need to compute expectations of random vectors.
Traditional Monte Carlo method may suffer from slow convergence.
In this work we propose to approximate random vectors by the Functional Tensor Train depending on $\xi$.

We start with a Cartesian product discretization.
We assume that the parametrizing random variables $\xi_1,\ldots,\xi_d$ are independent, so their probability density function factorizes, $\rho(\xi) = \rho_1(\xi_1) \cdots \rho_d(\xi_d)$.
For each $\xi_k$, $k=1,\ldots,d,$ we introduce a Gauss quadrature weighted with $\rho_k$, with $n$ nodes $\{\xi_k^j\}_{j=1}^n$ and weights $\{w_k^j\}_{j=1}^n$.
The expectation of any scalar continuous function is approximated by the combined quadrature
\begin{equation}\label{eq:tpq}
\mathbb{E}\left[f(\xi)\right] \approx \sum_{j_1=1}^n \cdots \sum_{j_d=1}^n w_1^{j_1} \cdots w_d^{j_d} f(\xi_1^{j_1}, \ldots, \xi_d^{j_d}).
\end{equation}
The convergence rate of this quadrature is bounded by the convergence rates of individual quadratures, and, if $f$ is sufficiently smooth, they all can be much faster (e.g. exponential) than the Monte Carlo convergence rate.
However, direct application of this quadrature suffers from the \emph{curse of dimensionality}: it involves the computation of $n^d$ evaluations of $f(\xi)$, which quickly becomes infeasible already for a moderate $d$.

\subsection{Tensor Train decomposition}
The Tensor Train (TT) decomposition~\cite{osel-tt-2011} alleviates this problem by approximating the \emph{tensor} of coefficients
$$
F_{j_1,\ldots,j_d} = f(\xi_1^{j_1}, \ldots, \xi_d^{j_d}), \quad j_k=1,\ldots,n, \quad k=1,\ldots,d,
$$
by the following expansion:
\begin{equation}\label{eq:tt0}
F_{j_1,\ldots,j_d} \approx \tilde F_{j_1,\ldots,j_d} := \sum_{s_0,\ldots,s_d=1}^{r_0,\ldots,r_d} F^{(1)}_{s_0,j_1,s_1} F^{(2)}_{s_1,j_2,s_2} \cdots F^{(d)}_{s_{d-1},j_d,s_d},
\end{equation}
where the factors $F^{(k)} \in \mathbb{R}^{r_{k-1} \times n \times r_k}$ are called \emph{TT cores}, and the ranges $r_k$ of the new summation indices are called \emph{TT ranks}, $r_k(\tilde F)$ if we want to emphasize that these are TT ranks of the TT decomposition $\tilde F$.
Before delving into the computation of the TT decomposition in Section~\ref{sec:cross}, let us first consider the storage complexity of the TT decomposition itself, and the computational complexity of the expectation in this format.
For convenience, we introduce the maximal TT rank $r:=\max_{k=0,\ldots,d} r_k$,
and, unless we address individual TT ranks, we will omit the word ``maximal'', and refer to $r$ as the \emph{TT rank} $r(\tilde F)$ of a TT decomposition $\tilde F$.
Note that the TT cores contain $\mathcal{O}(dnr^2)$ elements, which can be much fewer than $n^d$ elements in $F$, if the TT rank $r$ is small.
For example, the probability density function of independent random variables sampled on a Cartesian grid factorizes into a rank-1 TT decomposition:
if the random variables are independent, $\rho(\xi) = \rho_1(\xi_1) \cdots \rho_d(\xi_d)$, and, discretizing each marginal density, we obtain the TT core of coefficients $F^{(k)} \in \mathbb{R}^{1 \times n \times 1},$ $F^{(k)}_j = \rho_k(\xi_k^j).$
When approximating standard scalar functions,
we can always let $r_0=r_d=1$.
Other TT ranks can be nontrivial.
Nonetheless, there exist broad classes of smooth or weakly correlated functions that yield low TT ranks, such as polylogarithmic in the approximation error~\cite{rdgs-tt-gauss-2022,uschmajew-approx-rate-2013,khor-low-rank-kron-P1-2006}.

Our main usage of the structure~\eqref{eq:tt0} is the fast computation of the quadrature~\eqref{eq:tpq}.
Indeed, we can start with summing up just the last TT core,
$$
E^{(d)}_{s_{d-1},s_d} = \sum_{j=1}^n w_d^j F^{(d)}_{s_{d-1},j,s_d}, \qquad s_{d-1}=1,\ldots,r_{d-1}, \quad s_d=1,\ldots,r_d.
$$
Now assume we have a matrix $E^{(k+1)} \in \mathbb{R}^{r_k \times r_d}$.
To implement the sums over both $i_k$ and $s_{k}$ we sum the $k$th TT core weighted with $w_{k}$, and multiply the resulting matrix with $E^{(k+1)}$,
$$
E^{(k)}_{s_{k-1},s_d} = \sum_{s_{k}=1}^{r_{k}} \left(\sum_{j=1}^n w_k^j F^{(k)}_{s_{k-1},j,s_k}\right) E^{(k+1)}_{s_k,s_d}, \qquad s_{k-1}=1,\ldots,r_{k-1}, \quad s_d=1,\ldots,r_d.
$$
Continuing recursively (and recalling that $r_0=r_d=1$), we obtain the approximate expectation
$$
\mathbb{\tilde E}\left[f\right] := \sum_{j_1=1}^n \cdots \sum_{j_d=1}^n w_1^{j_1} \cdots w_d^{j_d} \tilde F_{j_1,\ldots,j_d} = E^{(1)}.
$$
Counting the summation sizes, we see that the complexity of the expectation computed this way is $\mathcal{O}(dnr^2)$ similarly to the number of unknowns in a TT decomposition.

However, we need to approximate \emph{vector} functions of spatial discretization coefficients such as $\mathbf{u}(\xi)$, or directly the tensor of fully discrete expansion coefficients~\eqref{eq:discr-vec},
where $j$ is replaced by $j_1,\ldots,j_d$ of the Cartesian quadrature,
$
u_{i,j} = U_{i,j_1,\ldots,j_d}.
$
Since we want to keep the spatial discretization and solver black box,
we employ the so-called \emph{block TT} decomposition~\cite{dkos-eigb-2014},
where instead of summing over $s_0$ we let it enumerate the index $i$ of spatial degrees of freedom,
\begin{equation}\label{eq:btt}
U_{i,j_1,\ldots,j_d} \approx \tilde U_{i,j_1,\ldots,j_d} := \sum_{s_1,\ldots,s_d=1}^{r_1,\ldots,r_d} \hat U^{(1)}_{i,j_1,s_1} U^{(2)}_{s_1,j_2,s_2} \cdots U^{(d)}_{s_{d-1},j_d,s_d}.
\end{equation}
Note that we denote the first TT core $\hat U^{(1)} \in \mathbb{R}^{\hat N \times n \times r_1}$ to emphasize that it contains the extra index $i$.

In addition to computing expectations, the TT decomposition admits fast interpolation.
Interpolating basis functions in space are assumed to be known from the finite element discretization.
In $\xi$, we introduce Lagrange interpolation polynomials $\phi^k_{j_k}(\xi_k)$ centered at the quadrature nodes $\xi_k^{j_k}$.
Now, we just need to sum the basis functions interpolated at each given $x$ and $\xi$ with only one TT core at a time,
$$
\hat u^{(1)}_{s_1}(x,\xi_1) := \sum_{j_1=1}^n \sum_{i=1}^{\hat N} \psi_i(x) \hat U^{(1)}_{i,j_1,s_1} \phi^1_{j_1}(\xi_1), \quad u^{(k)}_{s_{k-1},s_k}(\xi_k):=\sum_{j_k=1}^n U^{(k)}_{s_{k-1},j_k,s_k} \phi^k_{j_k}(\xi_k),
$$
followed by the matrix multiplication
\begin{equation}\label{eq:ftt}
u_h(x,\xi) = \hat u^{(1)}(x,\xi_1) u^{(2)}(\xi_2) \cdots u^{(d)}(\xi_d)
\end{equation}
calculated by a recursion similar to that used to compute expectations with $\mathcal{O}(dnr^2)$ complexity.
The expansion~\eqref{eq:ftt} generalizing the TT structure to continuous variables was called the Functional Tensor Train decomposition~\cite{Marzouk-stt-2016,Gorodetsky-ctt-2019}.
With a slight abuse of notations we can write $r(u_h)$ and $r_k(u_h)$ for the maximal/individual TT ranks of $\tilde U$.
Similarly, we can consider a semi-discrete TT decomposition
$$
\mathbf{u}(\xi) = \mathbf{u}^{(1)}(\xi_1) u^{(2)}(\xi_2) \cdots u^{(d)}(\xi_d),
$$
which is naturally extensible to the solution in Algorithm~\ref{alg:newton},
\begin{equation}\label{eq:btt3}
\begin{bmatrix}\mathbf{y}(\xi) \\ \mathbf{u}(\xi) \\ \boldsymbol\lambda(\xi) \end{bmatrix} = \begin{bmatrix}\mathbf{y}^{(1)}(\xi_1) \\ \mathbf{u}^{(1)}(\xi_1) \\ \boldsymbol\lambda^{(1)}(\xi_1) \end{bmatrix} u^{(2)}(\xi_2) \cdots u^{(d)}(\xi_d).
\end{equation}

\subsection{TT-Cross for the approximate Newton's solution}\label{sec:cross}
The family of TT \emph{Cross} approximation algorithms~\cite{ot-ttcross-2010,so-dmrgi-2011proc,dklm-tt-pce-2015,ds-parcross-2020} computes a TT decomposition from $\mathcal{O}(dnr^2)$ evaluations of the function $f(\xi)$ at adaptively chosen points $\xi$.
The first key ingredient is the \emph{Alternating} iteration,
similar to the methods of Alternating Least Squares~\cite{holtz-ALS-DMRG-2012} and Density Matrix Renormalization Group~\cite{white-dmrg-1993,schollwock-2005}.
Instead of computing the elements of all TT cores simultaneously,
we fix all but one cores, and restrict the relevant problem (such as Least Squares) to the elements of only one TT core at a time.
In the next step, we freeze the TT core we have just computed, and release the next core for the update, and so on.
A motivation for this alternation over the subsets of unknowns comes from the fact that the tensor $\tilde F$ represented by its TT decomposition~\eqref{eq:tt0} is \emph{linear} in the elements of one TT core at a time, and hence for example a convex optimization problem on the elements of the whole tensor $\tilde F$ remains a convex optimization problem on the elements of one TT core $F^{(k)}$.
In contrast, the optimization of all TT cores collectively is usually non-convex even if the initial cost function of $\tilde F$ was quadratic.

To compute TT cores we can use interpolation rather than optimization:
\begin{align}
\tilde F_{j_1,\ldots,j_d} & = F_{j_1,\ldots,j_d} \nonumber\\
\forall (j_1,\ldots,j_d) \in \mathbf{J}_k & :=\{(j_1^t,\ldots,j_d^t) \in \mathbb{R}^d: t=1,\ldots,r_{k-1}n r_k\},
\label{eq:ttinterp}
\end{align}
$\mathbf{J}_k$ is a set of $r_{k-1}n r_k$ tensor indices.
Note that $F^{(k)}$ contains the same number of unknowns.
In fact, \eqref{eq:ttinterp} is a linear equation on the elements of $F^{(k)}$.
Indeed, let
\begin{align}
 F^{<k}_{t,s_{k-1}} & = \sum_{s_0,\ldots,s_{k-2}=1}^{r_0,\ldots,r_{k-2}} F^{(1)}_{s_0,j_1^t,s_1} \cdots F^{(k-1)}_{s_{k-2},j_{k-1}^t,s_{k-1}}, \label{eq:iface-l} \\
 F^{>k}_{t,s_{k}} & = \sum_{s_{k+1},\ldots,s_d=1}^{r_{k+1},\ldots,r_d} F^{(k+1)}_{s_{k},j_{k+1}^t,s_{k+1}} \cdots F^{(d)}_{s_{d-1},j_d^t,s_d}, \label{eq:iface-r}\\
 F^{\neq k}_{t, s_{k-1} j_k s_k} & = F^{<k}_{t,s_{k-1}} \cdot \delta_{j_k,j_k^t} \cdot F^{>k}_{t,s_{k}}, \quad F^{\neq k} \in \mathbb{R}^{r_{k-1}n r_k \times r_{k-1}n r_k}, \label{eq:cross-frame}
\end{align}
where $\delta_{i,j}=1$ when $i=j$ and $0$ otherwise.
One can check that the linearity
\begin{equation}\label{eq:ttlin}
\tilde F_{j_1^t,\ldots,j_d^t} = \sum_{s_{k-1},j_k,s_k} F^{\neq k}_{t, s_{k-1} j_k s_k} F^{(k)}_{s_{k-1},j_k,s_{k}},
\end{equation}
holds, or, stretching $\tilde F$ and $F^{(k)}$ into vectors,
$
\mathrm{vec}(\tilde F(\mathbf{J}_k)) = F^{\neq k} \mathrm{vec}(F^{(k)}).
$

For efficient computation of \eqref{eq:iface-l} and \eqref{eq:iface-r} and index selection in the course of iterations over $k=1,\ldots,d$,
we restrict the structure of $\mathbf{J}_k$ to a Cartesian form
\begin{equation}\label{eq:set-kron}
 \mathbf{J}_{k} = J_{<k} \times [n] \times J_{>k},
\end{equation}
where
$$
[n]=\{1,\ldots,n\}, \quad J_{<k} = \{(j_1^t,\ldots,j_{k-1}^t)\}_{t=1}^{r_{k-1}}, \quad J_{>k} = \{(j_{k+1}^{t},\ldots,j_{d}^t)\}_{t=1}^{r_k}.
$$
Now $F^{<k} \in \mathbb{R}^{r_{k-1}\times r_{k-1}}$ and $F^{>k}\in \mathbb{R}^{r_{k}\times r_{k}}$ can be constructed only from $J_{<k}$ and $J_{>k}$, respectively.
Moreover, we can construct the next sets
\begin{equation}\label{eq:nest}
J_{<k+1} \subset J_{<k} \times [n], \qquad J_{>k-1} \subset [n] \times J_{>k}
\end{equation}
for the next iteration for $F^{(k+1)}$ or $F^{(k-1)}$ by selecting optimal rows from small matrices $\overline{F}^{\le k} \in \mathbb{R}^{r_{k-1} n \times r_k}$ and $\overline{F}^{\ge k} \in \mathbb{R}^{n r_k \times r_{k-1}}$ with elements
\begin{align}\label{eq:cross-ileft}
 \overline{F}^{\le k}_{s_{k-1} j_k, s_k} & = \sum_{t_{k-1}=1}^{r_{k-1}} F^{<k}_{s_{k-1},t_{k-1}} F^{(k)}_{t_{k-1},j_k,s_k}
 \quad \mbox{and} \\
 \overline{F}^{\ge k}_{j_k s_k, s_{k-1}} & = \sum_{t_{k}=1}^{r_k} F^{(k)}_{s_{k-1},j_k,t_k} F^{>k}_{t_k,s_k}.\label{eq:cross-iright}
\end{align}
The set $P \subset \{1,\ldots,r_{k-1}n\}$ of indices of rows extracted from $\overline{F}^{\le k}$, or $P \subset \{1,\ldots,nr_{k}\}$ for $\overline{F}^{\ge k}$, is chosen to reduce the condition number of the corresponding submatrix $F^{<k+1} = \overline{F}^{\le k}_{P,:}$ or $F^{>k-1} = (\overline{F}^{\ge k}_{P,:})^\top$,
and hence the condition number of \eqref{eq:cross-frame}.
We use the Maximum Volume (\emph{maxvol}) algorithm \cite{gostz-maxvol-2010} for this.
Similarly to the submatrices $F^{<k+1}$ and $F^{>k-1}$, the next sets of tensor indices are constructed by taking $p$th tuples from \eqref{eq:nest} for $p \in P$,
\begin{equation}\label{eq:set-sel}
J_{<k+1} = (J_{<k} \times [n])_P, \qquad J_{>k-1} = ([n] \times J_{>k})_P.
\end{equation}

The basic TT-Cross algorithm described above generalizes elegantly \cite{dklm-tt-pce-2015} to the block TT decomposition \eqref{eq:btt} and the Newton's iterate in Line~\ref{alg:newton:solve} of Algorithm~\ref{alg:newton}.
The block TT decomposition can contain the spatial index $i$ in any TT core, for example, the $k$-th core, denoted $\hat U^{(k)} \in \mathbb{R}^{r_{k-1} \times \hat N \times n \times r_k}$ in this case:
\begin{align}\label{eq:bttk}
\tilde U_{i,j_1,\ldots,j_d}  = \sum_{s_0,\ldots,s_d=1}^{r_0,\ldots,r_d} & U^{(1)}_{s_0,j_1,s_1} \cdots U^{(k-1)}_{s_{k-2},j_{k-1},s_{k-1}} \cdot \hat U^{(k)}_{s_{k-1},i,j_k,s_k} \\
& \cdot U^{(k+1)}_{s_{k},j_{k+1},s_{k+1}} \cdots U^{(d)}_{s_{d-1},j_d,s_d}. \nonumber
\end{align}
Note that all TT cores but $\hat U^{(k)}$ have the same shapes as generic TT cores in \eqref{eq:tt0},
and hence we can perform the majority of the TT-Cross computations \eqref{eq:iface-l}--\eqref{eq:cross-frame}, \eqref{eq:set-kron} and \eqref{eq:set-sel} for $U^{(k)}$ instead of $F^{(k)}$ verbatim.
The linear equation \eqref{eq:ttlin} looks different but consistent:
\begin{equation}
U_{i,j_1^t,\ldots,j_d^t} = \sum_{s_{k-1},j_k,s_k} U^{\neq k}_{t, s_{k-1} j_k s_k} \hat U^{(k)}_{s_{k-1},i,j_k,s_{k}},
\end{equation}
and solving it for all values of $i$ gives all the elements of the block TT core $\hat U^{(k)}$.
Now we can use the Principal Component Analysis to reduce $\hat U^{(k)}$ to a 3-dimensional TT core.
Reshaping $\hat U^{(k)}$ into a $r_{k-1}n \times r_k \hat N$ matrix $\mathbf{\hat U}^{(k)}$, we can compute its Singular Value Decomposition $\mathbf{\hat U}^{(k)} = \mathbf{U} \mathbf{S} \mathbf{V}^\top$, and remove the latter singular values in $\mathbf{S}$ and corresponding singular vectors in $\mathbf{U}$ and $\mathbf{V}$ to keep only $\tilde r_k \le \min(r_{k-1}n, r_k \hat N)$ principal components $\mathbf{U}_{\tilde r_k},\mathbf{S}_{\tilde r_k}$ and $\mathbf{V}_{\tilde r_k}$, such that $\|\mathbf{\hat U}^{(k)} - \mathbf{U}_{\tilde r_k},\mathbf{S}_{\tilde r_k}\mathbf{V}_{\tilde r_k}^\top\|_F \le \delta \|\mathbf{\hat U}^{(k)} \|_F$ for some desired error threshold $\delta>0$.
Now, $\mathbf{U}_{\tilde r_k}$ can be reshaped back to a three-dimensional tensor $U^{(k)} \in \mathbb{R}^{r_{k-1} \times n \times \tilde r_k}$, which becomes the new TT core,
and the new indices can be selected as shown in \eqref{eq:cross-ileft} -- \eqref{eq:set-sel}.
If the current TT-Cross iteration increases $k\rightarrow k+1$, the right singular vectors can be discarded, since the next TT core will be replaced by $\hat U^{(k+1)}$ in the next step.
Vice versa, in the backward iteration step $k\rightarrow k-1$, the leading right singular vectors are reshaped to form the new TT core.
This mechanism allows us to replace the TT ranks $r_k$ by the new TT ranks $\tilde r_k$ adaptively in the course of the iteration,
reducing the number of parameters of the algorithm to just one
error threshold $\delta.$
Finally, the iteration is stopped also when the relative Frobenius norm of the difference between $\hat U^{(k)}$ at the current and the previous iteration is below $\delta$ for all $k$.

\section{Numerical examples}\label{sec:num}
The numerical tests were implemented in Matlab based on the TT-Toolbox~\cite{tt-toolbox}
and run on a Intel Xeon E5-2640 v4 CPU.

\subsection{Elliptic PDE}\label{sec:elliptic}
In the first test we consider the following benchmark example, an adaptation of that in \cite{Surowiec-stability-2021,Surowiec-risk-2022}.
We solve \eqref{eq:obj_intro}, where:
\begin{align}\label{eq:elliptic}
J(y,\xi) & := \frac{1}{2} \|y(x,\xi) - y_d(x)\|_{L^2(D)}^2, \\
S(u) & := \frac{\alpha}{2}\|u(x,\xi)\|_{L^2(D)}^2, \\
\text{s.t.} \quad \nu(\xi) \Delta y(x,\xi) & = g(x,\xi) + u(x,\xi), \quad (x,\xi) \in D \otimes [-1,1]^d, \\
\nu(\xi) & = 10^{\xi_1(\omega)-2}, \quad g(x,\xi) = \frac{\xi_2(\omega)}{100},
\end{align}
and boundary conditions depending on two options for the physical domain $D$:
\begin{itemize}
 \item $D=(0,1)$, $d=4$, and
\begin{align*}
 y|_{x=0} & = -1-\frac{\xi_3(\omega)}{1000}, &
 y|_{x=1} & = -\frac{2+\xi_4(\omega)}{1000}
\end{align*}
\item $D=(0,1)^2$, $d=6$, and
\begin{align*}
 y|_{x_1=0} & = b_1(\xi) (1-x_2) + b_2(\xi) x_2, &
 y|_{x_2=1} & = b_2(\xi) (1-x_1) + b_3(\xi) x_1 \\
 y|_{x_1=1} & = b_4(\xi) (1-x_2) + b_3(\xi) x_2, &
 y|_{x_2=0} & = b_1(\xi) (1-x_1) + b_4(\xi) x_1, \\
 b_1(\xi) & = -1-\frac{\xi_3(\omega)}{1000}, &
 b_2(\xi) & = -\frac{2+\xi_4(\omega)}{1000}, \\
 b_3(\xi) & = -1-\frac{\xi_5(\omega)}{1000}, &
 b_4(\xi) & = -\frac{2+\xi_6(\omega)}{1000},
\end{align*}
\end{itemize}
the random vector $\xi(\omega) = (\xi_1(\omega), \ldots, \xi_d(\omega)) \sim \mathcal{U}(-1,1)^d$ is uniformly distributed (so $\rho(\xi)=2^{-d}$), and
the desired state $y_d(x) = -\sin(50 x/\pi)$.

We start with the one-dimensional physical domain, $D=(0,1)$.
The physical variable $x$ is discretized with a uniform grid with step size $h=1/1024$,
and each of the random variables $\xi_k$ is discretized with $n_{\xi}=17$ Gauss-Legendre grid points.
The first regularization parameter $\alpha=10^{-2}$ is fixed to the value in the original benchmark.
However, the shared sparsity regularization parameter $\beta$ is varied to study its effect on the control.
Firstly, we fix $\varepsilon=10^{-5}$ and the TT approximation threshold $\delta=10^{-5}$.
The control for different $\beta$ is shown in Figure~\ref{fig:ell-beta}.
As we expected, larger $\beta$ increase the fraction of the domain where both the mean and the variance of the solution are (nearly) zero, i.e. the sparsity is shared.
The measure of the set where the absolute mean control is below $10^{-4}$,
the number of iterations of Algorithm~\ref{alg:newton} required to drive the relative increment between the consecutive iterations of the solution below $10^{-5}$,
and the misfit of the state
are shown in Table~\ref{tab:beta}.
In addition, we show total CPU times collected in MATLAB R2021b on an Intel Xeon E5-2640 v4 CPU at 2.40 GHz.
\begin{table}[h!]
\centering
\caption{Elliptic example (1D), measure of sparsity, number of Newton's iterations to stopping tolerance $10^{-5}$, state misfit, and total CPU times for different $\beta.$}
\label{tab:beta}
\begin{tabular}{ccccc}
$\beta$   & $\int_0^1 \mathbf{1}[|\mathbb{E}[u_h]|<10^{-4}]dx$ & \#iterations & $\mathbb{E}\left[\|y_h - y_d\|_{L^2(D)}^2\right]$  & CPU time (s.) \\ \hline
$0$       & 0.000          & 2              &  0.0645   &  10.4  \\
$10^{-2}$ & 0.108          & 75             &  0.0798   &  256  \\
$10^{-1}$ & 0.575          & 341            &  0.1763   &  1176  \\
$1$       & 0.890          & 1370           &  0.4246   &  3351  \\
\end{tabular}
\end{table}
Clearly, larger $\beta$ make the Hessian approximation less accurate, and require more iterations of the (inexact) Newton's method.
Unsurprisingly, $\beta=0$ corresponding to the linear problem requires only one essential Newton's iteration (the second iteration is just checking the convergence).
The TT ranks during the TT-Cross algorithm remain between $4$ and $6$ for all iterations and $\beta$ in the considered range.
The TT ranks also stay in the same range as the TT truncation tolerance is varied between $10^{-3}$ and $10^{-8}$.
It is possible to rigorously establish these numerical observations.
\begin{proposition}
The solution
$
\left[
\mathbf{y}^\top,
\mathbf{u}^\top,
\boldsymbol\lambda^\top
\right]^\top
$
of this 1D elliptic PDE in the final iteration of the TT-Cross admits an exact TT decomposition of TT ranks not greater than $7$.
\end{proposition}
\begin{proof}
The proof is given in Appendix~\ref{sec:1drank}.
\end{proof}

Note that the elementwise square $u_{i,j}^2$ within $\mathcal{R}_{\varepsilon,h}(\mathbf{u})$ can be computed exactly in the TT format, followed by taking the square root of deterministic coefficients only.
In contrast, computing the exact $L^1$ norm would require to approximate the elementwise modulus $\mathbf{u}_+ = \left[|u_{i,j}|\right]$ in the TT format first, followed by the quadrature.
However, since the modulus is non-smooth, the TT ranks of $\mathbf{u}_+$ reach $59$ in this experiment, in contrast to at most $6$ for $\mathbf{u}$.

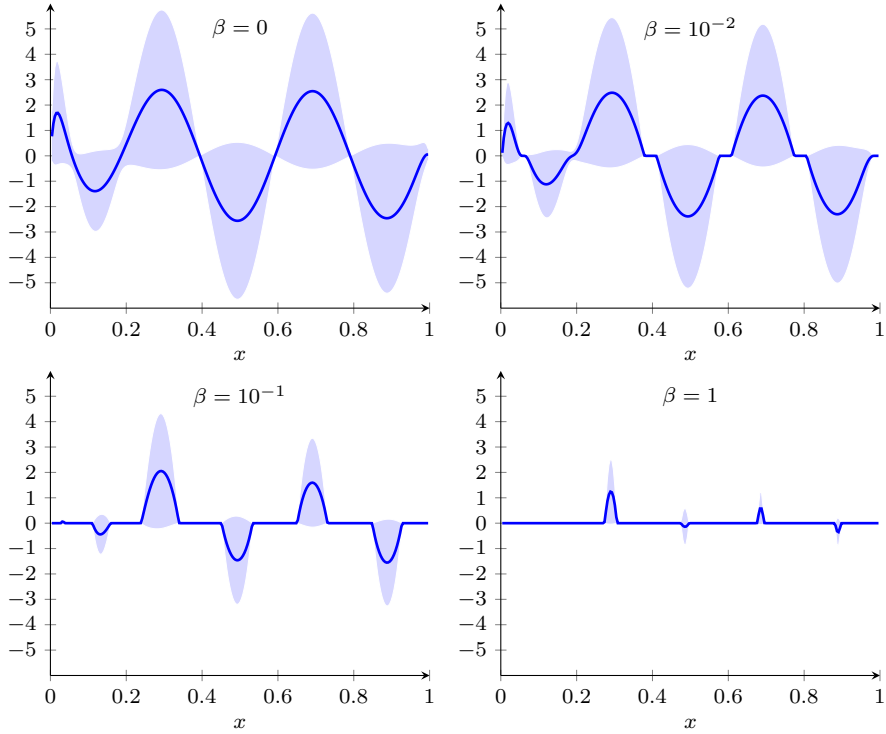
\begin{figure}[h!]
\noindent
\begin{tikzpicture}
 \begin{axis}[%
 width=0.42\linewidth,
 height=0.34\linewidth,
 xlabel=$x$,
 xmin=0,xmax=1,
 ymin=-6,ymax=6,
 ytick={-5,-4,-3,-2,-1,0,1,2,3,4,5},
 ]
 \addplot+[no marks,color=white,name path=minus] table[header=false,x index=0,y index=1]{elliptic1-umeanmstd-beta0-eps5.dat};
 \addplot+[no marks,color=white,name path=plus] table[header=false,x index=0,y index=1]{elliptic1-umeanpstd-beta0-eps5.dat};

\addplot[blue!16!white] fill between[of = minus and plus];

 \addplot+[blue,no marks,line width=1pt] table[header=false,x index=0,y index=1]{elliptic1-umean-beta0-eps5.dat};

 \node at (axis cs: 0.5,5) {$\beta=0$};
\end{axis}
\end{tikzpicture}
\begin{tikzpicture}
 \begin{axis}[%
 width=0.42\linewidth,
 height=0.34\linewidth,
 xlabel=$x$,
 xmin=0,xmax=1,
 ymin=-6,ymax=6,
 ytick={-5,-4,-3,-2,-1,0,1,2,3,4,5},
 ]
 \addplot+[no marks,color=white,name path=minus] table[header=false,x index=0,y index=1]{elliptic1-umeanmstd-beta-2-eps5.dat};
 \addplot+[no marks,color=white,name path=plus] table[header=false,x index=0,y index=1]{elliptic1-umeanpstd-beta-2-eps5.dat};

\addplot[blue!16!white] fill between[of = minus and plus];

 \addplot+[blue,no marks,line width=1pt] table[header=false,x index=0,y index=1]{elliptic1-umean-beta-2-eps5.dat};

 \node at (axis cs: 0.5,5) {$\beta=10^{-2}$};
\end{axis}
\end{tikzpicture}\\
\noindent
\begin{tikzpicture}
 \begin{axis}[%
 width=0.42\linewidth,
 height=0.34\linewidth,
 xlabel=$x$,
 xmin=0,xmax=1,
 ymin=-6,ymax=6,
 ytick={-5,-4,-3,-2,-1,0,1,2,3,4,5},
 ]
 \addplot+[no marks,color=white,name path=minus] table[header=false,x index=0,y index=1]{elliptic1-umeanmstd-beta-1-eps5.dat};
 \addplot+[no marks,color=white,name path=plus] table[header=false,x index=0,y index=1]{elliptic1-umeanpstd-beta-1-eps5.dat};

\addplot[blue!16!white] fill between[of = minus and plus];

 \addplot+[blue,no marks,line width=1pt] table[header=false,x index=0,y index=1]{elliptic1-umean-beta-1-eps5.dat};

 \node at (axis cs: 0.5,5) {$\beta=10^{-1}$};
\end{axis}
\end{tikzpicture}
\begin{tikzpicture}
 \begin{axis}[%
 width=0.42\linewidth,
 height=0.34\linewidth,
 xlabel=$x$,
 xmin=0,xmax=1,
 ymin=-6,ymax=6,
 ytick={-5,-4,-3,-2,-1,0,1,2,3,4,5},
 ]
 \addplot+[no marks,color=white,name path=minus] table[header=false,x index=0,y index=1]{elliptic1-umeanmstd-beta1-eps5.dat};
 \addplot+[no marks,color=white,name path=plus] table[header=false,x index=0,y index=1]{elliptic1-umeanpstd-beta1-eps5.dat};

\addplot[blue!16!white] fill between[of = minus and plus];

 \addplot+[blue,no marks,line width=1pt] table[header=false,x index=0,y index=1]{elliptic1-umean-beta1-eps5.dat};

 \node at (axis cs: 0.5,5) {$\beta=1$};
\end{axis}
\end{tikzpicture}
\caption{Elliptic PDE (1D), mean (solid lines) $\pm$ one standard deviation (shaded areas) of the optimal control $u_h(x,\xi)$ for different $\beta$ and $\varepsilon=10^{-5}.$}
\label{fig:ell-beta}
\end{figure}

Now we vary the smoothness parameter $\varepsilon$.
We fix $\beta=0.1$, $n_{\xi}=33$ (to make the parameter discretization error negligible), the TT approximation tolerance of $10^{-8}$, and compute the optimal control $u_{\varepsilon}$ for $\varepsilon$ in the range between $10^{-6}$ and $10^{-1}$.
In Figure~\ref{fig:ell-eps} (left), we show the
difference between the full control, its mean and variance computed at the given $\varepsilon$ compared to those computed at $\varepsilon=10^{-6}$,
as well as a similar difference in the total cost~\eqref{eq:obj_intro}.
We see that the convergence is linear in $\varepsilon$ for all outputs of the solution.
The effect of a high $\varepsilon$ is shown in Figure~\ref{fig:ell-eps} (right).
Essentially, $\varepsilon$ guides the threshold below which the solution should be seen as zero.

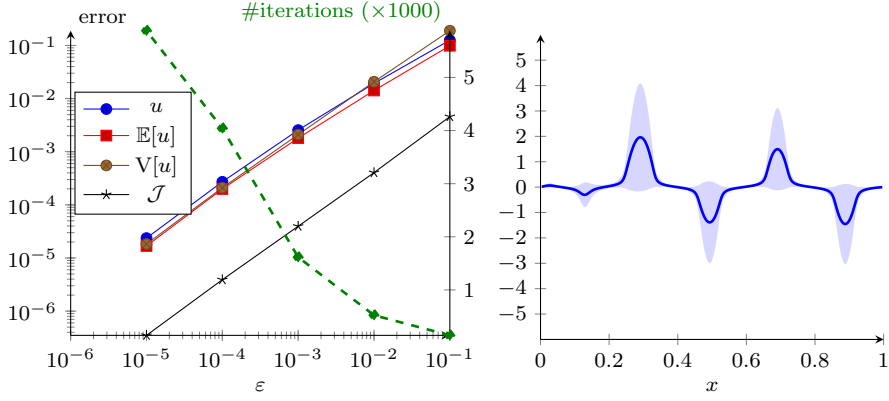
\begin{figure}[h!]
\noindent
\begin{tikzpicture}
 \begin{axis}[%
 width=0.42\linewidth,
 height=0.34\linewidth,
 axis y line=left,
 xlabel=$\varepsilon$,
 xmode=log,
 xmin=1e-6,
 ymode=log,
 ylabel=error,
 legend style={at={(0.01,0.80)},anchor=north west},
 ]
\addplot+ coordinates{
(1e-1, 1.2563e-01)
(1e-2, 1.9279e-02)
(1e-3, 2.5540e-03)
(1e-4, 2.7160e-04)
(1e-5, 2.3697e-05)
}; \addlegendentry{$u$};
\addplot+ coordinates{
(1e-1, 9.8590e-02)
(1e-2, 1.4173e-02)
(1e-3, 1.8040e-03)
(1e-4, 1.9890e-04)
(1e-5, 1.6840e-05)
}; \addlegendentry{$\mathbb{E}[u]$};
\addplot+ coordinates{
(1e-1, 1.8903e-01)
(1e-2, 2.0716e-02)
(1e-3, 2.0995e-03)
(1e-4, 2.0988e-04)
(1e-5, 1.8346e-05)
}; \addlegendentry{$\mathrm{V}[u]$};
\addplot+ coordinates{
(1e-1, 0.141618672 - 0.137048702) 
(1e-2, 0.137452900 - 0.137048702) 
(1e-3, 0.137088277 - 0.137048702) 
(1e-4, 0.137052615 - 0.137048702) 
(1e-5, 0.137049058 - 0.137048702) 
}; \addlegendentry{$\mathcal{J}$};
\end{axis}
 \begin{axis}[%
 width=0.42\linewidth,
 height=0.34\linewidth,
 axis y line=right,
 hide x axis,
 xmode=log,
 xmin=1e-6,
 ytick={1000,2000,3000,4000,5000},
 yticklabels={1,2,3,4,5},
 y label style={at={(1.0,1.0)},anchor=south east,rotate=-90},
 ylabel=\textcolor{green!50!black}{\#iterations ($\times 1000$)},
 ymode=normal,
 legend style={at={(0.01,0.99)},anchor=north west},
 ]
\addplot+[green!50!black,dashed,line width=1pt,mark=diamond*,mark options={fill=green!50!black}] coordinates{
(1e-1, 145 )
(1e-2, 524 )
(1e-3, 1616)
(1e-4, 4042)
(1e-5, 5879)
};
\end{axis}
\end{tikzpicture}
\begin{tikzpicture}
 \begin{axis}[%
 width=0.38\linewidth,
 height=0.34\linewidth,
 xlabel=$x$,
 xmin=0,xmax=1,
 ymin=-6,ymax=6,
 ytick={-5,-4,-3,-2,-1,0,1,2,3,4,5},
 ]
 \addplot+[no marks,color=white,name path=minus] table[header=false,x index=0,y index=1]{elliptic1-umeanmstd-beta-1-eps1.dat};
 \addplot+[no marks,color=white,name path=plus] table[header=false,x index=0,y index=1]{elliptic1-umeanpstd-beta-1-eps1.dat};

\addplot[blue!16!white] fill between[of = minus and plus];

 \addplot+[blue,no marks,line width=1pt] table[header=false,x index=0,y index=1]{elliptic1-umean-beta-1-eps1.dat};
\end{axis}
\end{tikzpicture}
\caption{Elliptic PDE (1D), $\beta=10^{-1}$. Left: errors in the full control $\|u_{\varepsilon} - u_{10^{-6}}\|_{L^2_{\rho}(\Xi; L^2(D))}$, mean control $\|\mathbb{E}[u_{\varepsilon}] - \mathbb{E}[u_{10^{-6}}]\|_{L^2(D)}$, variance of the control $\|\mathrm{Var}[u_{\varepsilon}] - \mathrm{Var}[u_{10^{-6}}]\|_{L^2(D)}$, and the total cost $\mathcal{J}_{\varepsilon} - \mathcal{J}_{10^{-6}}$, as well as the number of iterations till the stopping tolerance $10^{-8}$. Right: mean (solid line) $\pm$ one standard deviation (shaded area) of the optimal control for $\varepsilon=10^{-1}$.
}
\label{fig:ell-eps}
\end{figure}

\begin{figure}[h!]
\noindent
\begin{tikzpicture}
 \begin{axis}[%
 width=0.42\linewidth,
 height=0.34\linewidth,
 xlabel=\#iterations,
 xmode=normal,
 ymode=log,
 ymin=3e-6,ymax=1e-1,
 legend style={at={(0.99,0.99)},anchor=north east},
 ]
\addplot+ coordinates{
(16 , 8.1080e-02)
(43 , 2.5404e-02)
(124, 6.6489e-03)
(340, 1.7039e-03)
(880, 4.4788e-04)
}; \addlegendentry{$\varepsilon=10^{-5}$}; 
\addplot+ coordinates{
(43  , 2.3739e-02)
(120 , 5.8076e-03)
(310 , 1.1083e-03)
(659 , 1.5085e-04)
}; \addlegendentry{$\varepsilon=10^{-3}$}; 
\addplot+ coordinates{
(41 , 1.8012e-02)
(100, 2.9693e-03)
(191, 3.6071e-04)
(298, 3.7960e-05)
(409, 3.6430e-06)
}; \addlegendentry{$\varepsilon=10^{-2}$}; 
\end{axis}
\end{tikzpicture}
\begin{tikzpicture}
 \begin{axis}[%
 width=0.42\linewidth,
 height=0.34\linewidth,
 xlabel=$N_{PDE}$,
 xmode=log,
 ymode=log,
 ymin=3e-6,ymax=1e-1,
 legend style={at={(0.01,0.01)},anchor=south west},
 ]

\addplot+ coordinates{
(2^10*191, 1.1118e-02)
(2^12*191, 4.6762e-03)
(2^14*191, 2.0489e-03)
}; \addlegendentry{\footnotesize MC $10^{-5}$}; 

\addplot+ coordinates{
(2^10*100, 1.1483e-02)
(2^12*100, 5.6641e-03)
(2^14*100, 3.7030e-03)
}; \addlegendentry{\footnotesize MC $10^{-4}$}; 

\addplot+ coordinates{
(2^10*41, 2.0903e-02)
(2^12*41, 1.8901e-02)
(2^14*41, 1.8145e-02)
}; \addlegendentry{\footnotesize MC $10^{-3}$}; 

\addplot+[nodes near coords,point meta = explicit symbolic, every node near coord/.style={anchor=north west,inner sep=2pt}] coordinates{
(41 *1917, 1.8012e-02) 
(100*2012, 2.9693e-03)[$10^{-4}$]
(191*2050, 3.6071e-04)[$10^{-5}$]
(298*2056, 3.7960e-05)[$10^{-6}$]
(409*2062, 3.6430e-06) 

}; \addlegendentry{\footnotesize TT}; 

\addplot[domain=2e5:3e6, no marks,dashed] {5*x^(-1/2)}; \addlegendentry{\footnotesize $N_{PDE}^{-1/2}$};
\end{axis}
\end{tikzpicture}
\caption{Elliptic PDE (1D), left: mean control error $\|\mathbb{\tilde E}[u_{\varepsilon}] - \mathbb{E}_*[u_{\varepsilon}]\|_{L^2(D)}$ in the TT method with Newton's iterations for different smoothing parameters $\varepsilon$. Right: $\|\mathbb{\tilde E}[u_{\varepsilon}] - \mathbb{E}_*[u_{\varepsilon}]\|_{L^2(D)}$ in the Monte Carlo and TT methods with respect to the number of deterministic PDE solutions for $\varepsilon=10^{-2}$. Numbers below points in the TT plot and in the legend of MC plots show the stopping tolerance.  The reference solution $\mathbb{E}_*[u_{\varepsilon}]$ is computed with the TT method and tolerance $10^{-8}$.}
\label{fig:ell-mc}
\end{figure}
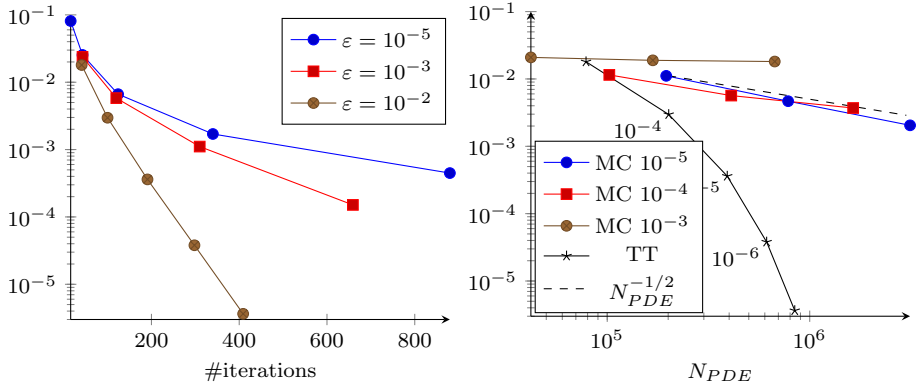

From the left plot of Figure~\ref{fig:ell-eps} we can also notice that the number of iterations grows asymptotically logarithmically with $\varepsilon$.
In Figure~\ref{fig:ell-mc} (left) we study this further.
For each $\varepsilon$, we take the final mean control computed with the TT approximation and stopping tolerance of $10^{-8}$ as a reference solution $\mathbb{E}_*[u_{\varepsilon}]$, and use it to estimate the solution error at previous iterations.
We see that eventually the method recovers the linear convergence, although the pre-asymptotic regime may take quite a few iterations for small $\varepsilon$.

In addition, in Figure~\ref{fig:ell-mc} (right) we compare the TT approach with the Monte Carlo method in terms of the computational complexity, measured in the number of deterministic PDE solves.
The Monte Carlo method uses the same approximate Newton's iteration, but the expectations are approximated by a Monte Carlo quadrature instead of the TT approximation with a Gaussian quadrature.
The TT method adapts the TT ranks (and hence the number of PDE solves needed during the TT-Cross algorithm) towards the stopping tolerance, so we just need to vary the latter.
However, for the Monte Carlo method, the stopping tolerance and the number of Monte Carlo samples are independent parameters, so we vary them both.
We see that the Monte Carlo method exhibits the usual $1/\sqrt{N}$ rate of convergence, much slower than the linear convergence of the TT method.

\begin{figure}[h!]
\noindent
\includegraphics[width=0.32\linewidth]{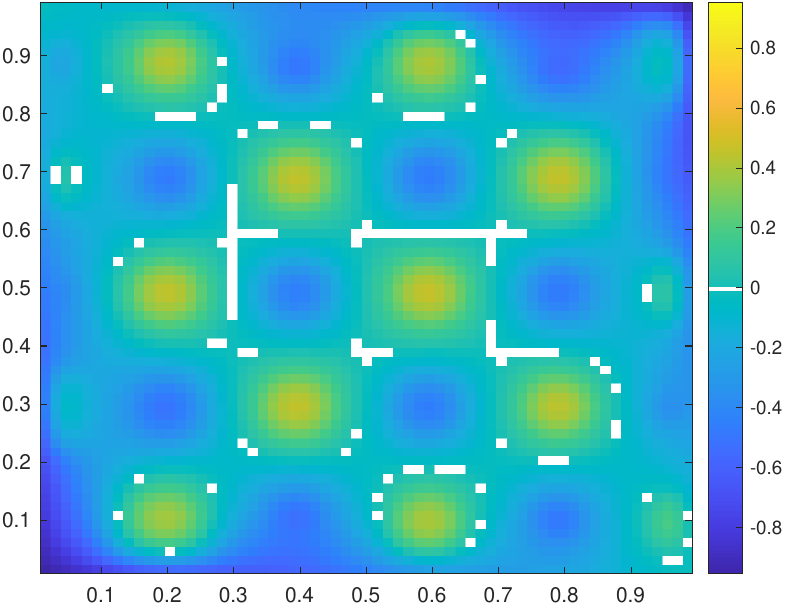}
\includegraphics[width=0.32\linewidth]{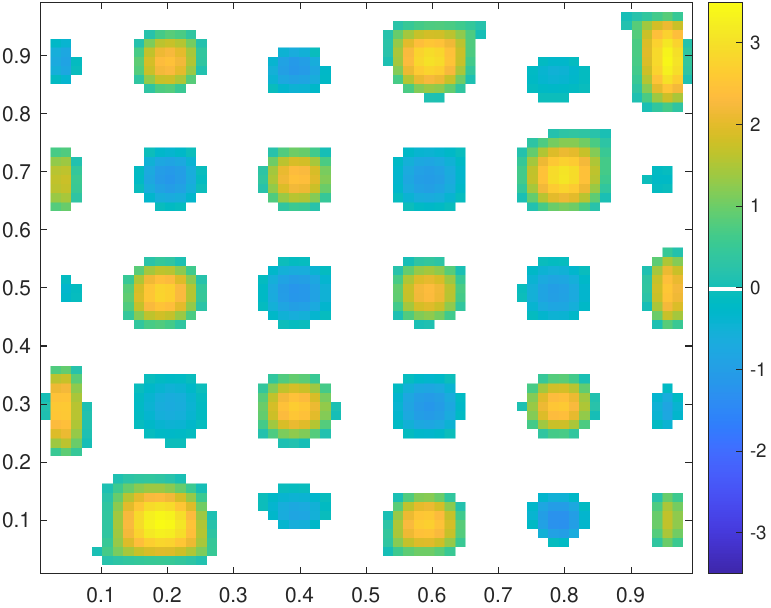}
\includegraphics[width=0.32\linewidth]{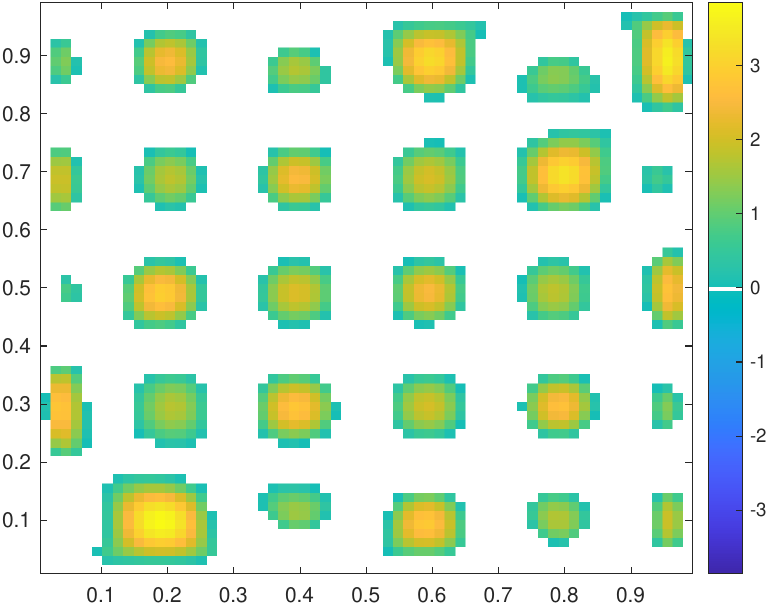} \\
\includegraphics[width=0.32\linewidth]{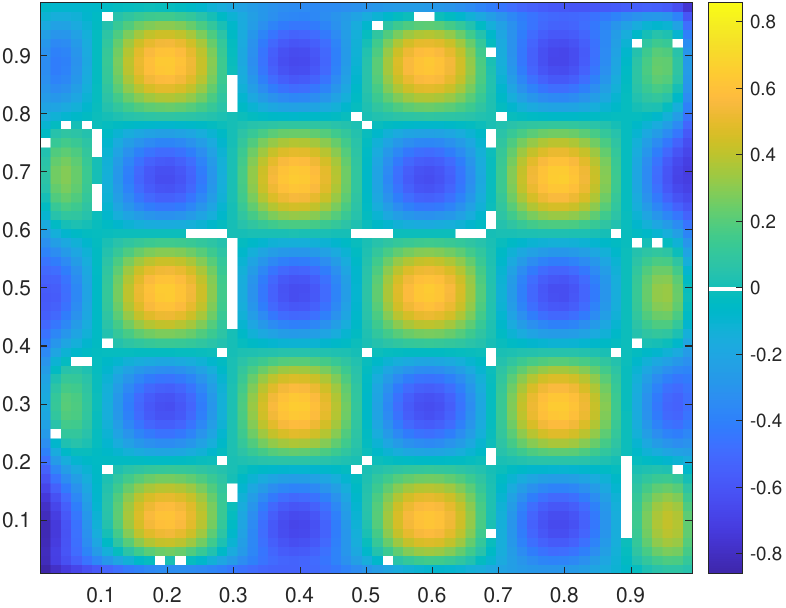}
\includegraphics[width=0.32\linewidth]{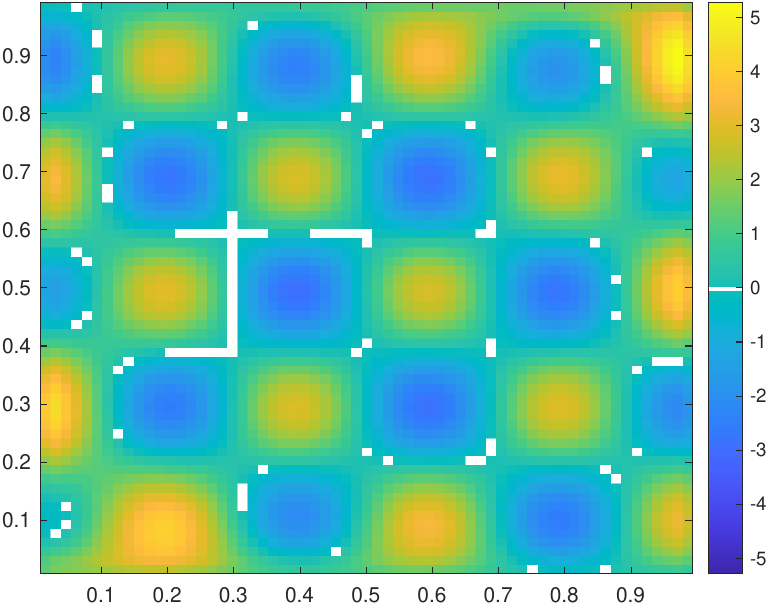}
\includegraphics[width=0.32\linewidth]{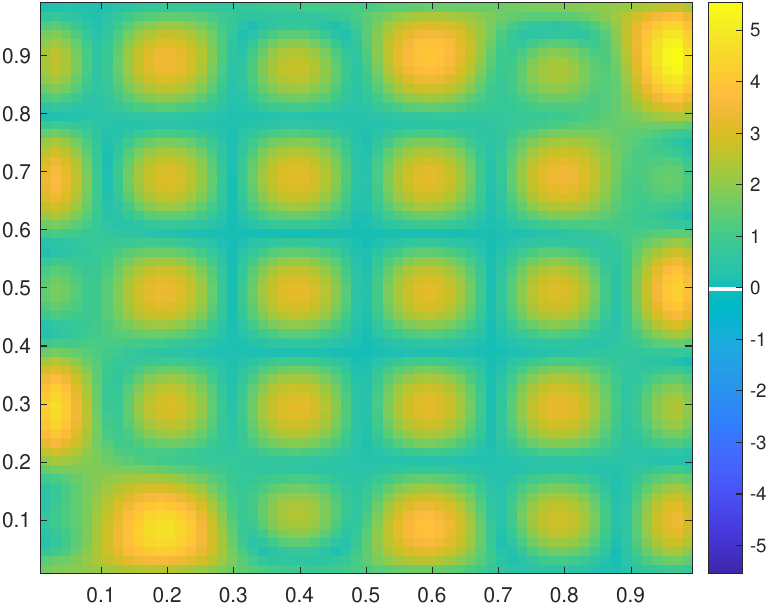}
\caption{Elliptic PDE (2D) with $h=1/64$. Left: $\mathbb{E}[y_h]$. Middle: $\mathbb{E}[u_h]$. Right: $\mathrm{Std}[u_h]$. Top: $\beta=0.1$ and $\varepsilon=10^{-5}$, $\mathbb{E}\left[\|y_h - y_d\|_{L^2(D)}^2\right] = 0.1843,$ $\int \mathbf{1}[|\mathbb{E}[u_h]|<10^{-4}]dx = 0.681.$ Bottom: $\beta=0$, $\mathbb{E}\left[\|y_h - y_d\|_{L^2(D)}^2\right] = 0.0925,$ $\int \mathbf{1}[|\mathbb{E}[u_h]|<10^{-4}]dx = 0.$}
\label{fig:2dsol}
\end{figure}
Figure~\ref{fig:2dsol} shows the results for the case when $D = (0,1)^2$, i.e., a
2-dimensional physical domain. A similar behavior as in the 1-dimensional setting is
observed.

\subsection{Topology optimization}\label{sec:topology}
In this section we consider the topology optimization problem under uncertain Young's modulus.
This problem aims to find an optimal distribution of the material of given relative volume $\bar V$ over the given domain $D$, which
minimizes the compliance of the structure.
The material distribution is encoded by the function of volume fraction $u(x,\xi): D \times \Xi  \rightarrow [0,1]$,
where $0$ corresponds to the absence of the material, and $1$ corresponds to the presence of the material (fractional values aid optimization and may also be seen as a reduced thickness of the material).
After discretization of the physical domain $D$, the compliance optimization problem can be written as
\begin{align}\label{eq:compliance}
\min_{\mathbf{u}} \ \mathbb{E}[C(\mathbf{u};\xi)], \ C(\mathbf{u};\xi) & := \mathbf{f}^\top \mathbf{y}(\xi), \\
\mbox{s.t.} \ K(\mathbf{u}; \xi) \mathbf{y}(\xi) & = \mathbf{f}, \nonumber \\
\int_D u_h(x,\xi)d x  & = \bar V \cdot |D|, & \mbox{a.s.} \label{eq:volconstr} \\
0 \le u_h(x,\xi) & \le 1, & \mbox{a.s.} \label{eq:matconstr}
\end{align}
where $\mathbf{u} \in \mathbb{R}^{\hat N N}$ is the vector of nodal values of the material volume fraction $u(x,\xi)$,
$K(\mathbf{u}; \xi) \in \mathbb{R}^{\hat N \times \hat N}$ is the stiffness matrix of the linear elasticity model
\begin{align*}
 -\nabla \cdot (E(u; x,\xi) \nabla y(x,\xi)) & = f(x)
\end{align*}
subject to appropriate boundary conditions,
$\mathbf{f} \in \mathbb{R}^{\hat N}$ is the force vector,
$\mathbf{y}(\xi)  \in \mathbb{R}^{\hat N}$ is the displacement vector,
and $C(\mathbf{u};\xi) \in \mathbb{R}$ is the compliance.
The net Young's modulus $E(u; x,\xi)$ depends on the material distribution, and a random Young's modulus of the pure material \cite{Keshavarzzadeh-top-2021},
\begin{align*}
E(u; x,\xi) & = E_{\min} + \tilde u^3 \left(E_0 - E_{\min} + 0.02 \sum_{k=1}^d \sqrt{\lambda_k} \phi_k(x) \xi_k\right),
\end{align*}
where $E_0$ is the mean Young's modulus of the pure material, perturbed by a truncated Karhunen-Loeve expansion (KLE) of a random field with uniformly distributed $\xi_k \sim \mathcal{U}(-1,1)$ and a Gaussian covariance function,
with eigenvalues $\lambda_k$ sorted descending and corresponding eigenfunctions $\phi_k(x)$.
Specifically, truncating the KLE at the total variance error of $1\%$, we obtain $d=8$ random variables.
$E_{\min} = 10^{-9}$ is the Young's modulus of the void (taken nonzero to avoid the singularity of the stiffness matrix),
and
$$
\tilde u = (I - \eta \Delta)^{-1} u
$$
is the filtered material distribution, with $\Delta$ being the Laplace operator with natural boundary conditions, and $\eta>0$ is a small parameter.
We use $\eta=0.1$ throughout the experiments.
Similarly to \cite{Keshavarzzadeh-top-2021,Torres-topology-2021}, we consider the MBB beam optimization problem where $D = [0, n_x] \times [0, n_y]$ with $n_x/n_y=4$, such that $\hat N=n_x n_y = 4 n_y^2$, whereas the mesh size is fixed to 1.
Thus, in the numerical experiments, we will choose $n_y$ and set $n_x= 4 n_y$.
Also, the correlation length of the KLE is set to $n_y$.
We will consider two meshes: a coarse one with $n_y=25$, and a fine one with $n_y=100$.
The right hand side $f(x)$ applies a unit force downwards at the top center point of the domain and is zero otherwise, the bottom corners of the domain are fixed, and the rest of it is free.

To make the material distribution more stable with respect to random perturbations, we consider two penalties.
Firstly, as in \cite{Torres-topology-2021}, we penalize the standard deviation of the compliance, $\mathrm{Std}[C(\mathbf{u};\xi)]$.
Secondly, we consider the smoothed $L^1$-norm of the material distribution, $\mathcal{R}_{\varepsilon,h}(\mathbf{u})$, enforcing the sparsity (that is, absence) of the material shared across the random realizations.
The smoothing parameter is fixed to $\varepsilon=10^{-3}$.
We introduce two regularization parameters $\kappa \ge 0$ and $\beta \ge 0$, controlling the weight of each penalty, so instead of \eqref{eq:compliance} we consider the following objective function:
\begin{equation}\label{eq:compliance_pen}
\mathbb{E}[C(\mathbf{u};\xi)] + \kappa \cdot \mathrm{Std}[C(\mathbf{u};\xi)] + \mathcal{R}_{\varepsilon,h}(\mathbf{u}) \, ,
\end{equation}
augmented by the spatial constraints \eqref{eq:volconstr}--\eqref{eq:matconstr}.
The implementation is based on the \texttt{top88} Matlab code \cite{Andreassen-top88-2011}.
Specifically, we seek a TT approximation of the discretization coefficients of the filtered solution $\mathbf{\tilde u}$.
Since the structure of the problem \eqref{eq:compliance}--\eqref{eq:matconstr} is challenging for Newton's method, we resort to the simpler
projected gradient descent method with a fixed step size $\tau>0$.
In its $i$-th iteration
we use the block TT-Cross outlined in Section~\ref{sec:cross} to approximate directly the next iterate,
\begin{align*}
\mathbf{u}^{(i)} & = (I - \eta \Delta_h)\mathbf{\tilde u}^{(i)}, \\
\mathbf{\tilde u}^{(i+1)} & = (I - \eta \Delta_h)^{-1} \mathcal{P}\left(\mathbf{u}^{(i)} - \tau \nabla_{\mathbf{u}} \left[\mathbb{E}[C(\mathbf{u}^{(i)};\xi)] + \kappa \cdot \mathrm{Std}[C(\mathbf{u}^{(i)};\xi)] + \mathcal{R}_{\varepsilon,h}(\mathbf{u}^{(i)})\right]\right),
\end{align*}
where $\Delta_h$ is the finite element discretization of the Laplace operator in the filter, and $\mathcal{P}$ is the projection onto the constraints \eqref{eq:volconstr}--\eqref{eq:matconstr}.
For simplicity, we just apply the two constraints sequentially: firstly \eqref{eq:volconstr} is satisfied by the corresponding linear transformation of the solution, followed by clamping the solution pointwise to ensure \eqref{eq:matconstr}. Although the second step may disturb \eqref{eq:volconstr} in general, we observed that, starting from the constant initial guess $\mathbf{u}^{(0)} \equiv \bar V$ and using a small step size $\tau$, \eqref{eq:matconstr} is satisfied up to a small margin, decreasing with iterations, and hence eventually both constraints become accurate. A more sophisticated linear programming is possible. Crucially, both constraints \eqref{eq:volconstr}--\eqref{eq:matconstr} apply pointwise in $\xi$, and can thus be implemented inside the TT-Cross algorithm.
In each gradient descent iteration, we carry out 1 TT-Cross iteration initialized with the previous gradient descent iterate, truncating singular values of the block TT cores to the relative error threshold $10^{-2}$.
To ensure convergence we also do continuation in $\kappa$, increasing it gradually, by starting with $\kappa=0$, and incrementing $\kappa$ by $2\cdot 10^{-3}$ each iteration until it reaches the desired value.

\medskip
\noindent
\textbf{Numerical tests}.
For faster computations, we consider first a coarser grid of size $100 \times 25$.
For this grid, we choose the gradient step size $\tau = 3\cdot 10^{-3}$, which is about a good balance between speed and stability, and carry out 5000 iterations. For performance metrics, we consider the mean TT rank over iterations and cores,
$$
\langle r \rangle = \frac{1}{5000} \sum_{i=1}^{5000} \frac{1}{d-1} \sum_{k=1}^{d-1} r_k(\mathbf{\tilde u}^{(i)}),
$$
and the standard deviation of the thresholded material distribution $u_h(x,\xi)>1/2$ and its spatial average,
\begin{equation}
\mathcal{S}(x) := \mathrm{Std}[\mathbb{I}_{u_h(x,\xi) > 1/2}], \qquad \mathcal{\bar S} := \int_D \mathcal{S}(x) dx.
\label{eq:std-material}
\end{equation}
Since the indicator function $\mathbb{I}(x,\xi)$ is non-smooth, it may lack an efficient TT approximation. Here, we compute the standard deviation in \eqref{eq:std-material} by the Monte Carlo method using $1024$ samples from $\mathcal{U}([-1,1]^d)$. However, a more accurate TT-driven coordinate transformation for approximation of an indicator function is possible \cite{cds-dirt-rare-2024}.

In Figures~\ref{fig:top_coarse} and \ref{fig:top_coarse_xhigh} we compare all four combinations of penalties, composed from two options for the standard deviation penalty ($\kappa=0$ or $\kappa=3$) and the shared sparsity penalty ($\beta=0$ or $\beta=1$).
In Figure~\ref{fig:top_coarse_xhigh}, we see that the standard deviation penalty leads actually to a more random structure.
In Figure~\ref{fig:top_ttranks} we show the evolution of TT ranks of the solution $\mathbf{\tilde u}$ as the iteration progresses.
All penalties reduce the TT ranks due to reduced uncertainty, and the best performing methods are those with $\beta=1$.
\begin{figure}[htb]
\centering
\caption{Mean optimized topology with different penalties. Left top: $\beta=\kappa=0$ (no penalty). Right top: $\beta=1$, $\kappa=0$. Left bottom: $\beta=0$, $\kappa=3$. Right bottom: $\beta=1$, $\kappa=3$.}
\label{fig:top_coarse}
\begin{minipage}{0.49\linewidth}
\centering
\fbox{\includegraphics[width=0.96\linewidth]{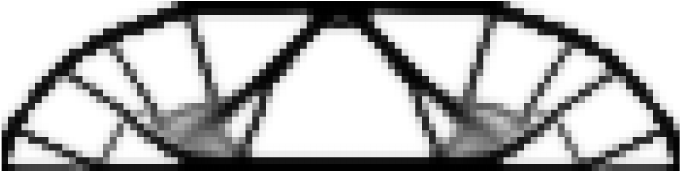}}

$\langle r\rangle = 7.88$ \quad $C=86.652\pm0.655$
\end{minipage}
\begin{minipage}{0.49\linewidth}
\centering
\fbox{\includegraphics[width=0.96\linewidth]{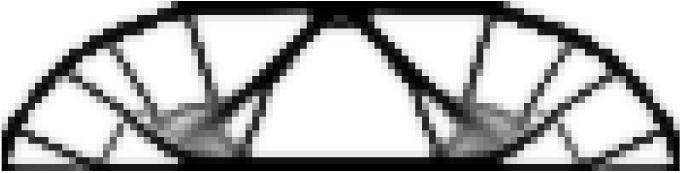}}

$\langle r\rangle = 1.39$ \quad $C=87.044\pm0.700$
\end{minipage}\\
\begin{minipage}{0.49\linewidth}
\centering
\fbox{\includegraphics[width=0.96\linewidth]{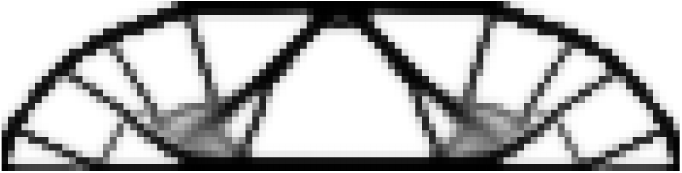}}

$\langle r\rangle = 2.03$ \quad $C=87.110\pm1.776$
\end{minipage}
\begin{minipage}{0.49\linewidth}
\centering
\fbox{\includegraphics[width=0.96\linewidth]{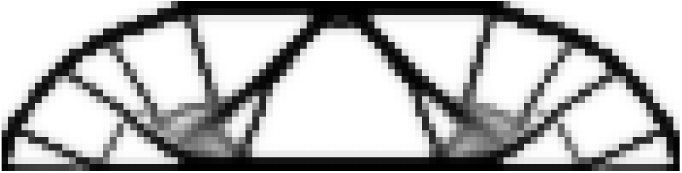}}

$\langle r\rangle = 2.33$ \quad $C=87.080\pm1.059$
\end{minipage}
\end{figure}
\begin{figure}[htb]
\centering
\caption{Standard deviation \eqref{eq:std-material} of the material distribution above $1/2$ with different penalties. Left top: $\beta=\kappa=0$ (no penalty). Right top: $\beta=1$, $\kappa=0$. Left bottom: $\beta=0$, $\kappa=3$. Right bottom: $\beta=1$, $\kappa=3$.}
\label{fig:top_coarse_xhigh}
\begin{minipage}{0.49\linewidth}
\centering
\fbox{\includegraphics[width=0.96\linewidth]{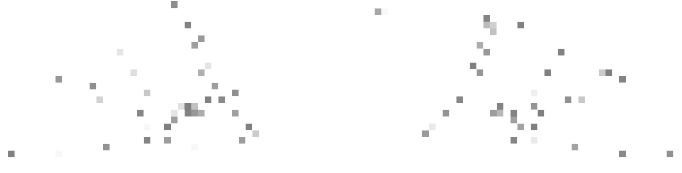}}

$\mathcal{\bar S} = $1.0276e-02
\end{minipage}
\begin{minipage}{0.49\linewidth}
\centering
\fbox{\includegraphics[width=0.96\linewidth]{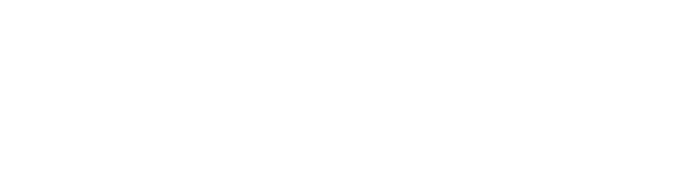}}

$\mathcal{\bar S} = $0
\end{minipage}\\
\begin{minipage}{0.49\linewidth}
\centering
\fbox{\includegraphics[width=0.96\linewidth]{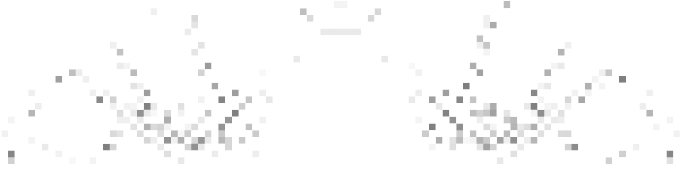}}

$\mathcal{\bar S} = $1.6618e-02
\end{minipage}
\begin{minipage}{0.49\linewidth}
\centering
\fbox{\includegraphics[width=0.96\linewidth]{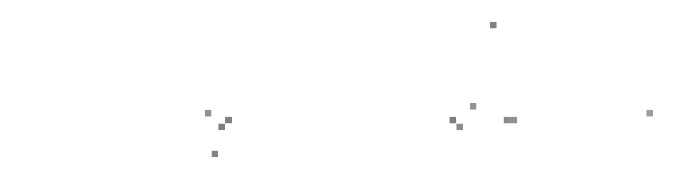}}

$\mathcal{\bar S} = $2.0142e-03
\end{minipage}
\end{figure}

\begin{figure}[htb]
\centering
\caption{Maximal TT ranks as a function of the iteration number}
\label{fig:top_ttranks}
\begin{tikzpicture}
\begin{axis}[%
xmode=log,
ymode=normal,
width=0.5\linewidth,
height=0.4\linewidth,
xlabel=iteration $(i)$,
ylabel=$\max_k r_k(\mathbf{\tilde u}^{(i)})$,
ymin=0,
legend style={at={(0.01,0.99)},anchor=north west},
]
\addplot+[no marks,dashed,line width=1pt] table[header=false,x index=0,y index=1]{coarse-ttranks.dat}; \addlegendentry{$\beta=0,\kappa=0$};
\addplot+[no marks] table[header=false,x index=0,y index=2]{coarse-ttranks.dat}; \addlegendentry{$\beta=1,\kappa=0$};
\addplot+[no marks,dotted,line width=2pt] table[header=false,x index=0,y index=3]{coarse-ttranks.dat}; \addlegendentry{$\beta=0,\kappa=3$};
\addplot+[no marks,line width=1pt] table[header=false,x index=0,y index=4]{coarse-ttranks.dat}; \addlegendentry{$\beta=1,\kappa=3$};
\end{axis}
\end{tikzpicture}
\end{figure}

Finally, we carry out the topology optimization experiment on a fine spatial grid of size $400\times 100$.
We increase the gradient step size to $\tau = 5 \cdot 10^{-2}$
to achieve the solution increments similar to those in the previous coarse-grid experiment,
which provides a well-converged solution after the same number of iterations of $5000$.
Since $\kappa>0$ did not improve sparsity or performance, we test only $\kappa=0$ here.
For $\beta=0$ (no penalty) the solution could not be completed due to the TT ranks exceeding $40$ after $893$ iterations, and the Matlab process crashing due to running out of 64Gb of memory.
For $\beta=1$ (enforcing shared sparsity), we managed to carry out $5000$ iterations, reaching the compliance $80.295\pm0.643$.
The TT rank has reached the maximum of $6$ actually in the first iterations, and then decreased, as shown in Figure~\ref{fig:top_fine} (right).
The mean material distribution is shown in Figure~\ref{fig:top_fine} (left).
The standard deviation of the thresholded solution averages to 3.5165e-05, which would give again an almost blank figure similarly to Fig.~\ref{fig:top_coarse_xhigh} (top right).
Again the penalty makes the material distribution more reproducible with respect to randomness.
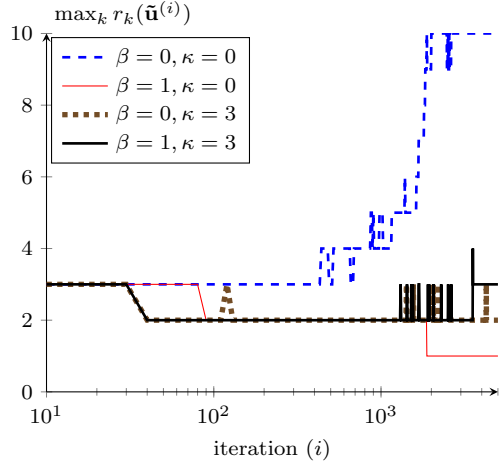
\begin{figure}[htb]
\centering
\caption{Optimized topology on a fine $400\times 100$ grid. Left: $\mathbb{E}[\mathbf{\tilde u}(\xi)]$  with $\beta=1$. Right: maximal TT rank as a function of the iteration number (dashed: $\beta=0$, solid: $\beta=1$).}
\label{fig:top_fine}
\begin{minipage}{0.49\linewidth}
\centering
\fbox{\includegraphics[width=0.96\linewidth]{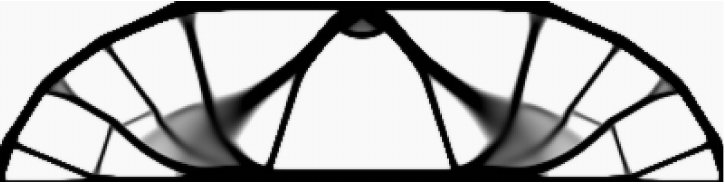}}

$\langle r\rangle = 1.01$, \quad $C=80.295\pm0.643$, \quad $\mathcal{\bar S} = $3.5165e-05
\end{minipage}
\begin{minipage}{0.49\linewidth}
\centering
\begin{tikzpicture}
\begin{axis}[%
xmode=log,
ymode=normal,
width=0.8\linewidth,
height=0.2\linewidth,
xlabel=iteration $(i)$,
ylabel=$\max_k r_k(\mathbf{\tilde u}^{(i)})$,
ymin=0,
ymax=15,
]
\addplot+[no marks,dashed,line width=1pt] table[header=false,x index=0,y index=1]{fine-beta0-ttranks.dat}; 
\addplot+[no marks] table[header=false,x index=0,y index=1]{fine-beta1-ttranks.dat}; 
\end{axis}
\end{tikzpicture}
\end{minipage}
\end{figure}

\section{Conclusion}
We have developed both first and approximate second order methods for the PDE-constrained optimization with a smoothed shared sparsity penalty.
For a nonzero smoothing parameter we obtain a linear convergence.
The error depends linearly on the smoothing parameter as well.
Smooth function approximations also converge sub-exponentially in the number of quadrature points in random parameters and TT ranks (exponentially if the function is analytic, and faster than any algebraic degree if the function is infinitely differentiable).
This makes the overall rate of convergence sub-exponential in the total computational cost.
This can be much faster than the algebraic rate of convergence of low-order and Monte Carlo methods,
as we demonstrated in the benchmark elliptic PDE example.
Moreover, a more structured solution with shared sparsity may actually exhibit lower TT ranks compared to the unconstrained solution.
This opens the way for the shared sparsity optimization in real-life applications, such as the topology optimization.

At the moment, we observe a linear convergence in $\varepsilon$, which is faster than the theoretically predicted rate of $\varepsilon^{-1/2}$.
It remains to a future research to obtain a sharp estimate of the convergence rate.

\appendix
\section{TT decomposition of the solution of the 1D elliptic PDE}\label{sec:1drank}
We are looking for the solution in the block TT format \eqref{eq:btt3},
$$
\begin{bmatrix}
\mathbf{y}(\xi) \\
\mathbf{u}(\xi) \\
\boldsymbol\lambda(\xi)
\end{bmatrix}
 =
\begin{bmatrix}
\mathbf{y}^{(1)}(\xi_1) \\
\mathbf{u}^{(1)}(\xi_1) \\
\boldsymbol\lambda^{(1)}(\xi_1)
\end{bmatrix}
u^{(2)}(\xi_2) u^{(3)}(\xi_3) u^{(4)}(\xi_4),
$$
as it is returned from the block TT cross.
The Gauss-Newton system corresponding to the elliptic PDE reads
\begin{align}\label{eq:kkt_app}
\begin{bmatrix}
 \mathbf{M}_y & 0 & \mathbf{A}(\xi_1) \\
 0 & \mathbf{M}_u & -\mathbf{B}^\top \\
 \mathbf{A}(\xi_1) & -\mathbf{B} & 0
\end{bmatrix}
\begin{bmatrix}
\mathbf{y}(\xi) \\
\mathbf{u}(\xi) \\
\boldsymbol\lambda(\xi)
\end{bmatrix}
=
\begin{bmatrix}
\mathbf{M}_y\mathbf{y}_d \\
0 \\
-\mathbf{g}(\xi_2) - \mathbf{b}_3(\xi_3) - \mathbf{b}_4(\xi_4)
\end{bmatrix},
\end{align}
where $\mathbf{M}_y$ is the mass matrix in state, independent of $\xi$,
$\mathbf{A}(\xi_1)$ is the symmetric stiffness matrix of $\nu(\xi_1) \Delta$ which depends only on $\xi_1$,
$\mathbf{M}_u$ is the control part of the approximate Hessian of the Lagrangian (independent of $\xi$),
$\mathbf{B}$ is the actuator matrix independent of $\xi$,
so is the desired state $\mathbf{y}_d$,
and $\mathbf{b}_3(\xi_3)$ and $\mathbf{b}_4(\xi_4)$ are the right hand sides accommodating left and right boundary conditions, respectively, each of which depends on either $\xi_3$ or $\xi_4$.
We split the state into 4 components: $\mathbf{y}(\xi) = \mathbf{y}_u(\xi) + \mathbf{y}_g(\xi) + \mathbf{y}_3(\xi) + \mathbf{y}_4(\xi)$, which satisfy the following equations:
\begin{align}\label{eq:yu_app}
\mathbf{A}(\xi_1) \mathbf{y}_u(\xi) & = \mathbf{B} \mathbf{u}(\xi), & \mathbf{A}(\xi_1) \mathbf{y}_g(\xi) & = -\mathbf{g}(\xi_2)\\
\mathbf{A}(\xi_1) \mathbf{y}_3(\xi) & = -\mathbf{b}_3(\xi_3), & \mathbf{A}(\xi_1) \mathbf{y}_4(\xi) & = -\mathbf{b}_4(\xi_4).
\end{align}
Since $\mathbf{A}(\xi_1)^{-1}$ acts linearly on a function independent of $\xi_1$,
we obtain rank-1 TT decompositions
\begin{align*}
\mathbf{y}_g(\xi) & = \mathbf{y}_g^{(1)}(\xi_1) y_g^{(2)}(\xi_2), & \mathbf{y}_3(\xi) & = \mathbf{y}_3^{(1)}(\xi_1) y_3^{(3)}(\xi_3), & \mathbf{y}_4(\xi) & = \mathbf{y}_4^{(1)}(\xi_1) y_4^{(4)}(\xi_4).
\end{align*}
From the first equation of \eqref{eq:kkt_app} and \eqref{eq:yu_app} we get
$$
\boldsymbol\lambda(\xi) = \mathbf{A}(\xi_1)^{-1} \left(\mathbf{M}_y\mathbf{y}_d - \mathbf{M}_y \mathbf{A}(\xi_1)^{-1} \mathbf{B} \mathbf{u}(\xi) - \mathbf{M}_y \mathbf{y}_g(\xi) - \mathbf{M}_y \mathbf{y}_3(\xi) - \mathbf{M}_y \mathbf{y}_4(\xi)\right),
$$
whereas from the second equation
$$
\mathbf{u}(\xi)  = \mathbf{M}_u^{-1} \mathbf{B}^\top \boldsymbol\lambda(\xi),
$$
which gives us the Schur complement depending only on $\xi_1$,
\begin{align*}
\underbrace{\left(\mathbf{I} + \mathbf{A}(\xi_1)^{-1} \mathbf{M}_y A(\xi_1)^{-1} \mathbf{B} \mathbf{M}_u^{-1} \mathbf{B}^\top\right)}_{\mathbf{S}(\xi_1)} \boldsymbol\lambda(\xi) & = \mathbf{A}(\xi_1)^{-1}\mathbf{M}_y \left(\mathbf{y}_d -  \mathbf{y}_g(\xi) -  \mathbf{y}_3(\xi) -  \mathbf{y}_4(\xi)\right).
\end{align*}
Defining
\begin{align*}
\mathbf{Q}(\xi_1) & = \mathbf{S}(\xi_1)^{-1} \mathbf{A}(\xi_1)^{-1}\mathbf{M}_y,
\end{align*}
we obtain the Lagrange multiplier in the explicit form
\begin{align*}
\boldsymbol\lambda(\xi) & = \mathbf{Q}(\xi_1) \left(\mathbf{y}_d -  \mathbf{y}_g(\xi) -  \mathbf{y}_3(\xi) -  \mathbf{y}_4(\xi)\right).
\end{align*}
Each summand in the right hand side of $\boldsymbol\lambda(\xi)$ is a rank-1 TT decomposition,
so $\boldsymbol\lambda(\xi)$ is a TT decomposition of ranks not greater than 4,
$$
\boldsymbol\lambda(\xi) = \boldsymbol\lambda^{(1)}(\xi_1) \lambda^{(2)}(\xi_2) \lambda^{(3)}(\xi_3) \lambda^{(4)}(\xi_4).
$$
So are $\mathbf{u}(\xi)$ and $\mathbf{y}_u(\xi)$, which moreover differ from $\boldsymbol\lambda(\xi)$ only in the first TT cores.
Finally, adding rank-1 $\mathbf{y}_g(\xi)$, $\mathbf{y}_3(\xi)$ and $\mathbf{y}_4(\xi)$,
we obtain that $\mathbf{y}(\xi)$ is a TT decomposition of ranks not greater than 7.
Specifically,
\begin{align*}
\begin{bmatrix}
\mathbf{y}(\xi) \\
\mathbf{u}(\xi) \\
\boldsymbol\lambda(\xi)
\end{bmatrix} & =
\begin{bmatrix}
\mathbf{A}(\xi_1)^{-1} \mathbf{B} \mathbf{M}_u^{-1} \mathbf{B}^\top  \boldsymbol\lambda^{(1)}(\xi_1) \\
\mathbf{M}_u^{-1} \mathbf{B}^\top  \boldsymbol\lambda^{(1)}(\xi_1) \\ \boldsymbol\lambda^{(1)}(\xi_1)
\end{bmatrix}
\lambda^{(2)}(\xi_2) \lambda^{(3)}(\xi_3) \lambda^{(4)}(\xi_4) \\
& +
\begin{bmatrix}
\mathbf{y}_g^{(1)}(\xi_1)\\
 0 \\
 0
\end{bmatrix}
 y_g^{(2)}(\xi_2)
 +
\begin{bmatrix}
\mathbf{y}_3^{(1)}(\xi_1)\\
 0\\
 0
\end{bmatrix}
 y_3^{(3)}(\xi_3) +
\begin{bmatrix}
\mathbf{y}_4^{(1)}(\xi_1)  \\
0 \\
0
\end{bmatrix}
y_4^{(4)}(\xi_4).
\end{align*}

\section*{Declarations}
\subsection*{Ethics approval and consent to participate}
Not applicable.
\subsection*{Consent for publication}
Not applicable.
\subsection*{Funding}
HA is partially supported by NSF grant DMS-2408877, the AirForce Office of Scientific Research under Award NO: FA9550-22-1-0248, and Office of Naval Research (ONR) under Award NO: N00014-24-1-2147.
SD is thankful for the support from Engineering and Physical Sciences Research
Council (EPSRC) New Investigator Award EP/T031255/1.
\subsection*{Availability of data and materials}
Not applicable.
\subsection*{Competing interests}
The authors declare that they have no competing interests.
\subsection*{Authors' contributions}
HA, SD, AO designed the study, performed the numerical experiments, analysed the results and wrote the manuscript. All authors read and approved the final manuscript.
\subsection*{Acknowledgements}
Not applicable.

\bibliographystyle{spmpsci}
\bibliography{refs}
\end{document}